\documentclass{article}
\usepackage{geometry}
\usepackage[utf8]{inputenc}
\usepackage{hyperref}
\usepackage{tikz}
\usepackage{amsmath,amsthm,amsfonts,amssymb,color,cite,xcolor,authblk}
\usepackage{diagbox,graphicx}
\usepackage{algorithm}
\usepackage{algpseudocode}
\usepackage[title]{appendix}
\usepackage{float}

\newtheorem{theorem}{Theorem}

\newtheorem{lemma}[theorem]{Lemma}

\newtheorem{observation}[theorem]{Observation}

\makeatletter
\newtheorem*{rep@theorem}{\rep@title}
\newcommand{\newreptheorem}[2]{%
\newenvironment{rep#1}[1]{%
 \def\rep@title{#2 \ref{##1}}%
 \begin{rep@theorem}}%
 {\end{rep@theorem}}}
\makeatother

\newreptheorem{theorem}{Theorem}

\theoremstyle{definition}

\newtheorem{remark}[theorem]{Remark}
\newtheorem{example}[theorem]{Example}

\newcommand{\Z}{\mathbb{Z}}
\newcommand{\E}{\mathcal{E}}

\newcommand{\D}{\mathcal D}

\begin{document}
\title{Optimization Tools for Computing Colorings of $[1,\dots,n]$ with\\ 
Few  Monochromatic Solutions  on $3$-variable Linear Equations}


\author[1]{Jes\'us A. De Loera}
\author[1]{Denae Ventura}
\author[1]{Liuyue Wang}
\author[2]{William J. Wesley}
\affil[1]{University of California, Davis}
\affil[2]{University of California, San Diego}



\date{\today}

\maketitle

\begin{abstract}
A famous result in arithmetic Ramsey theory says that for many linear homogeneous equations $\E$ there is a threshold value $R_k(\E)$ (the Rado number of $\E$) such that for any $k$-coloring of the integers in the interval $[1,n]$, with $n \ge R_k(\E)$, there exists at least one monochromatic solution. But one can further ask, \emph{how many monochromatic solutions is the minimum possible in terms of $n$?} Several authors have estimated this function before, here we offer new tools from integer and semidefinite optimization that help find either optimal or near optimal 2-colorings minimizing the number of monochromatic solutions of several families of 3-variable non-regular homogeneous linear equations. In the last part of the paper we further extend to three and more colors for the Schur equation, improving earlier work.
\end{abstract}

\section{Introduction}

An equation $\E$ is \emph{regular} if for all positive integers $r$, every $r$-coloring of $\mathbb{N}$ contains at least one monochromatic solution to $\E$. Richard Rado \cite{RadoThesis} proved that a linear homogeneous equation is regular if and only if a nonempty subset of its coefficients sums to zero. But in fact more is known, and a bound to the minimum number of monochromatic solutions can always be found. More precisely, a famous result of Frankl, Graham, and R\"odl \cite{FranklGrahamRodl} states that for any regular homogeneous linear equation ${\cal E}$ in $k$ variables, no matter what $r$-coloring is given to the integers from 1 to $n$, there will always be at least order $n^{k-1}$ monochromatic solutions of the equation $\cal E$. The goal of this paper is to investigate the same enumerative problem but for certain classes of homogeneous linear equations that are not necessarily regular.

Let $a\leq b$ be positive integers and let $[a,b]$ denote the set of integers $a, a+1, \cdots , b$. If $a=1$ we may use $\left[b\right]$ or $[1,b]$ depending on the context. Note that throughout the paper we often abbreviate $[n]=[1,..., n]$. Let $\mathcal{C}_k$ be the set of all $k$-colorings of $[n]$. 
Given a $k$-coloring $\chi$ of $[n]$, the function $\mu_{\chi}({\cal E},n,k)$ denotes the number of monochromatic solutions to the equation $\cal E$ found using the coloring $\chi$. It is important to note that if $(x,y,z)$ and $(y,x,z)$ are integer solutions to ${\cal E}$ with $x\neq y$, we consider them to be different and count them as two solutions. Let $R_k ({\cal E})$ denote the \emph{Rado number} of ${\cal E}$, which is the minimum integer $N$ such that any $k$-coloring of $[N]$ contains at least one monochromatic solution to ${\cal E}$. 
Let $T_{\cal E} (n)$ be the number of all integer solutions to the equation ${\cal E}$ in $[n]$. We denote by $M_{\cal E}(n,k)$ the function given by the minimum number of monochromatic solutions of equation $\cal E$ over all $k$-colorings of $[n]$. When $k=2$, we simply write $M_{\cal E}(n)$. Namely 

$$M_{\cal E} (n,k)=\min_{\chi \in \mathcal{C}_k } \mu_{\chi}({\cal E},n,k).$$ 

Now, despite the theorem of Frankl, Graham and R\"odl, the situation is more complicated for general equations, namely those that are not regular or not even homogeneous. For instance, there is a $4$-coloring of $\mathbb{N}$ that entirely avoids monochromatic solutions to the (nonregular) equation $x + y = 3z$. 
Thus the key question remains, \emph{what to do for non-regular equations?}
In this regard, Costello and Elvin \cite{CostelloElvin} made recent progress by showing that $M_\E(n) = \Omega (n^2)$ for equations $\E$ of the form $ax + ay = cz$, $a,c \in \mathbb{N}$, and the same authors conjectured that  if $\E$ is a 3-variable equation of the form $ax+by=cz$ with $a,b,c \in \mathbb{N}$, then $M_{{\cal E}}(n)= \Omega(n^2)$. 

\subsection*{Our Contributions}

The main goal of our paper is to compute $\mu_\chi({\cal E},n,2)$ for certain $2$-colorings $\chi$ of $[n]$ and bound the function $M_{\cal E}(n,2)$ when $\cal E$ is a 3-variable, not necessarily regular, homogeneous linear equation. We also contribute some improved bounds for Schur's equation when using more than two colors. More precisely here
are our results:

In Section \ref{theupperbounds}, we use integer programming to provide optimal 2-colorings and $M_{\cal E}(n,2)$ for fixed small values of $n$ and fixed equation $\E$. Some colorings had very clear patterns which we were able to extend theoretically for larger values of $n$ and give upper bounds for $M_{\cal E}(n,2)$. Nevertheless, some of the computer generated colorings were not as predictable, and the optimal colorings are in fact not necessarily unique, thus we sometimes designed the new coloring based on theoretical observations. 
We prove the following theorems (see Section \ref{theupperbounds}).


We start by giving bounds for the families of equations of the form $ax+ay=z$ and $ax-ay = z$.

\begin{theorem} \label{upperbax+ay}
Let $a$ be a positive integer and ${\cal E}$ be the equation $ax+ay=z$. The following holds 
$$M_{\cal E}(n) \leq  \left \lfloor\frac{n^2}{2a^4}+\frac{n}{2a^2}\right\rfloor,$$ but the bound is not tight.
\end{theorem}



Next we sharpen the results of \cite{CostelloElvin} for the family of equations $ax-ay=z$.

\begin{theorem} \label{upperbax-ay}

Let $a$ be a positive integer and ${\cal E}$ be the equation $ax-ay=z$. 
The following holds $$M_{\cal E}(n) \leq  \left\lfloor\frac{n^2}{a^3}-\frac{n^2}{2a^4}-\frac{n}{2a^2}\right\rfloor,$$

but the bound is not tight. 
\end{theorem}

We studied equations of the form $x+y=az$ when $a>2$ for which we provide different colorings. We show one general coloring and a different coloring when $a$ is odd, which improves the number of monochromatic solutions for $a = 3,5$.

\begin{theorem}\label{upperbx+y=az}
Let ${\cal E}$ be the equation $x+y=az$. 
The following holds

\[ M_{\cal E}(n)\leq \begin{cases}
     \dfrac{n^2}{4a^2} +O(n), & a=3,5, \\
     \\
      \dfrac{8(2a - 1)n^2}{a^4(4 + a)} + O(n),& a=4, a\ge 6,\\
   \end{cases}
\]

but this bound is not tight.
\end{theorem}




We also explore the more general equation $ax+by=z$ when $gcd(a,b)=1$ and give a general upper bound for $M_{\cal E}(n)$. 

\begin{theorem} \label{upperbax+by}
    Let $a$ and $b$ be positive integers with $a<b$ and $gcd(a,b)=1$. Let ${\cal E}$ be an equation of the form $ax+by=z$ and let $n\geq R_2 ({\cal E})$. 
    The following holds

    $$M_{\cal E}(n)\leq \sum_{z=a+b}^{\left\lfloor\frac{n}{a(a+b)}\right\rfloor} \left( \frac{z}{ab} - \left\{ \frac{b^{-1}z}{a} \right\} - \left\{\frac{a^{-1}z}{b} \right\} +1 \right),$$

where $\{x\}=x-\lfloor x \rfloor$, $b^{-1}$ is any integer such that $b^{-1}b\equiv 1\pmod a$ and $a^{-1}$ is any integer such that $a^{-1}a\equiv 1 \pmod b$. This bound is not tight.
    
\end{theorem}

Section \ref{methodsofIPSDP} contains an explanation of the ways we find optimal colorings for fixed values of $[1,\dots,n]$ using integer programming techniques. Theorems \ref{upperbax-ay}, \ref{upperbx+y=az}, and \ref{upperbax+by} were inspired by these optimal colorings from small values of $n$. We show some of these colorings in Appendix \ref{amyIPexperiments}.

Of course, instead of solving an integer program, we could solve a semidefinite program (SDP), but that often does not give a very good coloring. In Section \ref{jacktracesdp} we use a semidefinite programming method, previously used by Parrilo, Robertson, and Saracino for the equation $x + y = 2z$ in \cite{ParriloRobertsonSaracinoMonochromaticAP} and more recently by Roberston in \cite{Robertson_2DimSchur}. This allows us to give lower bounds on $M_{\cal E}(n,2)$ in the following situation:

\begin{theorem} \label{Theorem_ax_ay_z}
Let ${\cal E}:ax-ay = z$ and $n\geq R_2 ({\cal E})$. For $3\le a \le 7$, the following holds 
$$M_{\cal E}(n) \ge \frac{(2a-1)n^2}{2a^4}.$$ 
\end{theorem}

Combined with Theorem \ref{upperbax-ay}, we have that $M_{\cal E}(n) = \frac{2a-1}{2a^4}n^2(1+o(1))$ for ${\cal E} :ax-ay = z$ and $ 3 \le a \le 7$. In these special cases our bound coincides with the conjectured optimal bounds for the more general equation $ax-ay = bz$ from \cite{ButlerCostelloGrahamConstellations} and \cite{thanatipanonda2016minimum}. 
Moreover, we obtain a nontrivial, but not tight, bound for $M_\E(n)$ for the non-regular equation $x + y = 3z$. 
\begin{theorem} \label{Theorem_x_y_3z}
Let ${\cal E}:x+y=3z$ and $n\geq R_2 ({\cal E})$. The following bound holds

$$M_{\cal E} (n)\geq \frac{167n^2}{28800}.$$
\end{theorem}








We conclude the paper in Section \ref{sectionmorethan2colors} with improvements to one of the most important equations in the history of the subject, the \emph{Schur equation} $x+y = z$ for more than 2 colors. 


We will denote the number $M_{x+y = z}(n,k)$ by $M_{\text{Schur}}(n,k)$.   For the case $k =2$, the exact value is known, and three independent proofs that $M_{\text{Schur}}(n,2) = \frac{n^2}{11} + O(n)$ were given by Robertson and Zeilberger \cite{RobertsonZeilbergerSchurTriples}, Schoen \cite{SchurTriplesSchoen1999}, and Datskovsky \cite{SchurTriplesDatskovsky2003}. Note that some authors count solutions differently, i.e., only solutions where $x\le y$, which leads to the apparent discrepancy of a factor of 2 in the calculation of $M_{\text{Schur}}(n,k)$. In this work we do not require $x \le y$, and thus the highest order term in $M_{\text{Schur}}(n,k)$ is $\frac{n^2}{11}$ rather than $\frac{n^2}{22}$. 

A natural next step is to determine values of $M_{\text{Schur}}(n,k)$ for $k > 2$. This question, along with generalizations of the problem for the equation $x+ay = z$, was explored by Thanatipanonda \cite{ThanatipanondaSchurTriplesProblem2009}, who used a greedy algorithm to find $k$-colorings giving few monochromatic Schur 
triples for $k = 3,4,5$. The colorings gave the following upper bounds for $M_{\text{Schur}}(n,k)$: 
\begin{align*}
    M_{\text{Schur}}(n,3) &\le \frac{47n^2}{3119} + O(n)\\ M_{\text{Schur}}(n,4) &\le \frac{69631222699293042329481527n^2} {33538492045698352404702657699
}+O(n), \\ M_{\text{Schur}}(n,5) &\le \frac{2n^2}{7610.0730}+O(n)
\end{align*}

The algorithm given in \cite{ThanatipanondaSchurTriplesProblem2009} becomes slow for $k \ge 6$ and gives rational numbers of large height (that are suppressed in the bound for $M_{\text{Schur}}(n,5)$). Our final contribution is new bounds for small values of $k$. 

\begin{theorem}\label{TheoremMuUpperBounds}
    The following improved bounds on $M_{\text{Schur}}(n,k)$ hold:

    $$M_{\emph{Schur}}(n,3) \le \frac{n^2}{67} + O(n), \hspace{20 pt } M_{\emph{Schur}}(n,4) \le \frac{n^2}{496} + O(n).$$
\end{theorem}

\section{Upper Bounds on $M_{\cal E}(n)$ for 3-term linear equations} \label{theupperbounds}

In this section, we show upper bounds of $M_{\cal E}(n)$ for equations ${\cal E}$ of the forms $ax+ay=z$, $ax-ay=z$, $x+y=az$, and $ax+by=z$ when $gcd(a,b)=1$. We used integer programming to obtain the best colorings for fixed values of $n$. This helped us, in some cases, give a general coloring for all $n$. In the cases where our experiments did not show any predictable patterns in the colorings, we proceeded to design an appropriate coloring using theoretical observations. 

In our arguments we had three main types of colorings. One of them is the \emph{block coloring} where intervals of at least three integers receive the same color (Figure \ref{coloring_block_x+y=az}). We also used a coloring where all integers that are \emph{multiples} of a certain number $a$ receive the same color and the rest of the integers receive a different color (Figure \ref{coloring_ax+ay=z_ax-ay=z}). Finally, we used a coloring that assigns a certain color based on the congruence class$\pmod a$ of each integer. For example, coloring $\chi_1$ in Figure \ref{coloring_x+y=az} assigns red to all integers $i\equiv 0,2,5\pmod6$ in $[n]$ and the rest of the integers are blue. Some of our experiments show a mix of these types of colorings. We will write $k$-colorings with colors $0,1,\dots,k-1$ and use $i^m$ to denote a string of $m$ consecutive integers of color $i$.

In some parts we will abbreviate by $G_{\cal E}(n,a)$ the value of an upper bound of $M_{\cal E}(n)$. Similarly, we will use $G_{\cal E}(n,a,b)$ to abbreviate the value of an upper bound of $M_{\cal E}(n)$ when ${\cal E}$ is a three-term linear equation that depends on positive integers $a$ and $b$.\\

\begin{reptheorem} {upperbax+ay}
Let $a$ be a positive integer and ${\cal E}$ be the equation $ax+ay=z$. The following holds 
$$M_{\cal E}(n) \leq  \left \lfloor\frac{n^2}{2a^4}+\frac{n}{2a^2}\right\rfloor,$$ but the bound is not tight.
\end{reptheorem}

\begin{proof}[Proof of Theorem \ref{upperbax+ay}]

    Let ${\cal E}:ax+ay=z$ and $a\geq 2$. We prove that $\mu_\chi({\cal E},n,2)= \frac{n^2}{2a^4}$ where $\chi$ is the $2$-coloring of $[n]$ that assigns integers $i\equiv 0\pmod a$ the color red and everything else the color blue. See Figure \ref{coloring_ax+ay=z_ax-ay=z}.
    We count the monochromatic solutions by looking at each integer $z^{*}$ in $[n]$ that can take the value of $z$ in $ax+ay=z$ and according to its color we count the pairs of integers $(x^{*},y^{*})$ of the same color that satisfy $ax^{*}+ay^{*}=z^{*}$.

    Note that any such $z^{*}$ must be a multiple of $a$. Therefore, there cannot be any blue solutions as there are no multiples of $a$ in blue. It suffices to look only at all integer multiples of $a$. Given a red $z^{*}$, and a pair $x^* = ak_1$ and $y^* = ak_2$ that satisfy $ax^{*}+ay^{*}=z^{*}$, notice that $z^{*}$ will be forced to be a multiple of $a^2$, namely

    $$ax^{*}+ay^{*}= a(ak_1) + a(ak_2) = a^2 (k_1 + k_2 ) = z^{*},$$

    which in turn implies that $\frac{z^{*}}{a}$ is a multiple of $a$,

    $$x^{*}+y^{*}= \frac{z^{*}}{a}.$$

    Therefore it is equivalent to find pairs $(x^*,y^*)$ which are multiples of $a$ that sum to $\frac{z^*}{a}$, which is also a multiple of $a$. Note that $ \frac{z^{*}}{a} \leq \frac{n}{a}$, which implies that $x^*, y^* \leq \frac{n}{a}$. 
    We can now take each multiple of $a$, $\frac{z^{*}}{a}$, and count all pairs of multiples of $a$ that sum to it, which is simply $\frac{z^{*}}{a^2}$. 

    $$\sum_{z^* = a^2 k,\hspace{2pt} 1\leq k \leq \frac{n}{a}} \frac{z^*}{a^2} = \sum_{i=1}^{\lfloor\frac{n}{a^2}\rfloor} i \leq \left\lfloor\frac{n^2}{2a^4} + \frac{n}{2a^2}\right\rfloor.$$

Hence, $M_{\cal E}(n) \leq \left\lfloor \frac{n^2}{2a^4}+\frac{n}{2a^2}\right\rfloor$.

 \begin{figure}
    \centering
\begin{tikzpicture}

\draw (0,0.1)--(0,-0.1) node[black,below]{0};
\draw [blue, ultra thick] (0,0)--(1,0);
\draw [blue, ultra thick] (1,0)--(2,0);
\draw [blue, ultra thick] (2,0)--(3,0);
\filldraw [black] (4.3,-.4) circle (1pt);
\filldraw [black] (4.5,-.4) circle (1pt);
\filldraw [black] (4.7,-.4) circle (1pt);
\draw [blue, ultra thick] (3,0)--(15,0);
\draw [red, ultra thick] (1,0.1)--(1,-0.1) node[black,below]{$a$};
\draw [red, ultra thick] (2,0.1) --  (2,-0.1)node[black,below]{$2a$};
\draw [red, ultra thick] (3,0.1) --  (3,-0.1)node[black,below]{$3a$};
\draw [red, ultra thick] (6,0.1) --  (6,-0.1)node[black,below]{$ak$};
\draw [red, ultra thick] (15,0.1) --  (15,-0.1)node[black,below]{$\lfloor\frac{n}{a}\rfloor a$};
\filldraw [black] (7,-.4) circle (1pt);
\filldraw [black] (7.2,-.4) circle (1pt);
\filldraw [black] (7.4,-.4) circle (1pt);

\end{tikzpicture}
\caption{Coloring $\chi$ used in Theorems \ref{upperbax+ay} and \ref{upperbax-ay} where all multiples of $a$ are red and the rest is blue.}
\label{coloring_ax+ay=z_ax-ay=z}
 \end{figure}
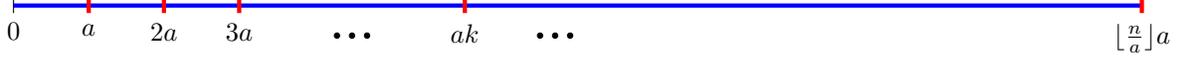

To prove that this bound is not tight, we conducted experiments using integer programming where we can appreciate a few discrepancies when $a=2,3,4$. We used values of $n$ greater than the corresponding Rado number, which is $R_2 (ax+ay=z)=4a^3+a$, \cite{HarborthMaasberg}. In Table \ref{table:4x+4y=z}, we list the results for $a=4$ and leave the rest in Appendix \ref{amyIPexperiments}. Let ${\cal E}: ax+ay=z$ and $G_{\cal E}(n,a) = \left\lfloor \frac{n^2}{2a^4}-\frac{n}{2a^2}\right\rfloor$.

\begin{table}[H]
\centering
\begin{tabular}{|p{0.1\textwidth}|p{0.15\textwidth}|p{0.15\textwidth}|p{0.15\textwidth}|p{0.15\textwidth}|}
           \hline
           $n$&    260&  300& 350& 400\\
           \hline
           $M_{\cal E}(n)$&   $1$&  $1$& $1$& $2$\\
           \hline
           $G_{\cal E}(n,4)$ & $132$ & $175$ & $239$& $312$ \\
           \hline
          $T_{\cal E}(n)$ &  $1056$&  $1406$&  $1892$&  $2500$\\
           \hline
           \textbf{Optimal Coloring}&  $0^{8}1^{56}0^{196}$&  $0^{10}1^{64}010^{3}$ $10^{220}$&  $1^{11}0^{84}(1000)^{63}10^{2}$& $1^{14}0^{98}(1000)^{72}1$\\
           \hline

\end{tabular}
\caption{Results using integer programmingfor equation ${\cal E} : 4x+4y=z$.}
\label{table:4x+4y=z}
\end{table}

This concludes the proof.
\end{proof}

\begin{reptheorem} {upperbax-ay}

Let $a$ be a positive integer and ${\cal E}$ be the equation $ax-ay=z$. 
The following holds $$M_{\cal E}(n) \leq  \left\lfloor\frac{n^2}{a^3}-\frac{n^2}{2a^4}-\frac{n}{2a^2}\right\rfloor,$$

but the bound is not tight. 
   
\end{reptheorem}

\begin{proof}[Proof of Theorem \ref{upperbax-ay}]

    Let ${\cal E}:ax-ay=z$ and $a\geq 2$. We prove that $\mu_\chi({\cal E},n,2)= \frac{n^2}{a^3}-\frac{n^2}{2a^4}-\frac{n}{2a^2}$ where $\chi$ is the $2$-coloring of $[n]$ that assigns integers $i\equiv 0\pmod a$ the color red and everything else the color blue. See Figure \ref{coloring_ax+ay=z_ax-ay=z}. We count the monochromatic solutions by looking at each integer $z^{*}$ in $[n]$ that can take the value of $z$ in $ax-ay=z$ and according to its color we count the pairs of integers $(x^{*},y^{*})$ of the same color that satisfy $ax^{*}-ay^{*}=z^{*}$. Note that any such $z^{*}$ must be a multiple of $a$. Therefore, there cannot be any blue solutions as there are no multiples of $a$ in blue. It suffices to look only at all integers multiples of $a$.

    Given a value of $z^{*}$, and a pair $x^* = ak_1$ and $y^* = ak_2$ that satisfy $ax^{*}-ay^{*}=z^{*}$, notice that $z^{*}$ will be forced to be a multiple of $a^2$, namely

    $$ax^{*}-ay^{*}= a(ak_1) - a(ak_2) = a^2 (k_1 - k_2 ) = z^{*},$$

    which in turn implies that $\frac{z^{*}}{a}$ is a multiple of $a$,

    $$x^{*}-y^{*}= \frac{z^{*}}{a}.$$

    We take each value $z^*$ and count all pairs $(x^*,y^*)$ which are multiples of $a$ whose difference is $\frac{z^*}{a}$. All possible values of $\frac{z^*}{a}$ lie in the set $\{a,2a, \cdots, a\cdot\frac{n}{a^2} \}$ because $\frac{z^*}{a} \leq \frac{n}{a}$. 
    
    Note that there are $\frac{a(\frac{n}{a})-a}{a}$ ways to pick two multiples of $a$ whose difference is $a$. There are $\frac{a(\frac{n}{a})-2a}{a}$ ways to pick two multiples of $a$ whose difference is $2a$. Continuing this way, we finish by counting $\frac{a(\frac{n}{a})-\frac{n}{a}}{a}$ ways to pick two multiples of $a$ whose difference is $\frac{n}{a}$. Putting together these values we get

    $$\sum_{i=1}^{\frac{n}{a^2}}\left(\frac{n}{a}-i\right) \leq \left\lfloor\frac{n^2}{a^3}-\frac{n^2}{2a^4}-\frac{n}{2a^2}\right\rfloor.$$

    Hence, $M_{\cal E}(n)\leq \left\lfloor\frac{n^2}{a^3}-\frac{n^2}{2a^4}-\frac{n}{2a^2}\right\rfloor$ holds.

    To prove that this bound is not tight, we carried out experiments using integer programming where we noticed discrepancies when $a=3,4$. We used values of $n$ greater than the corresponding Rado number, which is $R_2 (ax-ay=z)= a^2$ (see \cite{HarborthMaasberg}). 
    See Table \ref{table:3x-3y=z} for a full comparison. Experiments for the case $a=4$ can be found in Table \ref{table:4x-4y=z} in Appendix \ref{amyIPexperiments}. Let $G_{\cal E}(n,a)=\left\lfloor\frac{n^2}{a^3}-\frac{n^2}{2a^4}-\frac{n}{2a^2}\right\rfloor$.


\begin{table}[H]
\centering
\begin{tabular}
{|p{0.1\textwidth}|p{0.15\textwidth}|p{0.15\textwidth}|p{0.15\textwidth}|p{0.15\textwidth}|}
           \hline
           $n$&    25&  50& 75& 100\\
           \hline
           $M_{\cal E}(n)$ &   $13$ &  $65$& $164$& $297$\\
           \hline
           $G_{\cal E}(n,3)$ & $17$ & $74$ & $169$& $303$ \\
           \hline
          $T_{\cal E}(n)$ &  $164$&  $664$&  $1550$&  $2739$\\
           \hline
           \multicolumn{5}{|c|}{ \textbf{Optimal Coloring:} for integers $i\leq n$, color $0$ if $i\equiv 0 \pmod 3$, and $1$ otherwise.}\\
           \hline
      
\end{tabular}
\caption{Results for ${\cal E}: 3x-3y=z$.}
\label{table:3x-3y=z}
\end{table}

\end{proof}

The same experiments when ${\cal E}: 5x-5y=z$ show no discrepancies with the corresponding values given by the upper bound in Theorem \ref{upperbax-ay} when $n=25,50,75$ and $100$. See Table \ref{table:5x-5y=z} for full comparison. In Section \ref{jacktracesdp}, we determine the exact value for $M_{\cal E}(n)$ when $3\leq a \leq 7$ and verify our results.


\begin{table}[H]
\centering
\begin{tabular}{|p{0.1\textwidth}|p{0.15\textwidth}|p{0.15\textwidth}|p{0.15\textwidth}|p{0.15\textwidth}|}
           \hline
           $n$&    $25$&  $50$& $75$& $100$\\
           \hline
           $M_{\cal E}(n)$&   $4$&  $17$& $39$& $70$\\
           \hline
           $G_{\cal E}(n,5)$&   $4$ &  $17$& $39$& $70$\\
           \hline
          $T_{\cal E}(n)$&  $110$&  $445$&  $1005$&  $1790$               \\
           \hline
           \multicolumn{5}{|c|}{ \textbf{Optimal Coloring:} for integers $i\leq n$, color $0$ if $i\equiv 0 \pmod 5$, and $1$ otherwise.}\\
           \hline
\end{tabular}
\caption{Results for ${\cal E}: 5x-5y=z$.}
\label{table:5x-5y=z}
\end{table}


 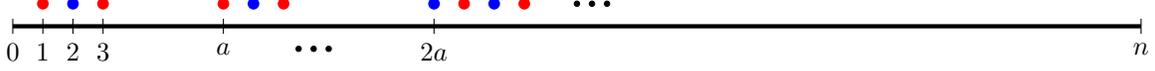
\begin{figure}
    \centering
\begin{tikzpicture}

 \draw (0,0.1)--(0,-0.1) node[black,below]{$0$};
 \draw [black, ultra thick] (0,0)--(15,0);
 \draw (15,0.1)--(15,-0.1) node[black,below]{$n$};

\draw (0.4, 0.1)--(0.4,-0.1) node[black,below]{$1$};
\draw (0.8, 0.1)--(0.8,-0.1) node[black,below]{$2$};
\draw (1.2, 0.1)--(1.2,-0.1) node[black,below]{$3$};
\draw (2.8, 0.1)--(2.8,-0.1) node[black,below]{$a$};
\draw (5.6, 0.1)--(5.6,-0.1) node[black,below]{$2a$};

\filldraw [black] (3.8,-.3) circle (1pt);
\filldraw [black] (4,-.3) circle (1pt);
\filldraw [black] (4.2,-.3) circle (1pt);

\filldraw [black] (7.5,.3) circle (1pt);
\filldraw [black] (7.7,.3) circle (1pt);
\filldraw [black] (7.9,.3) circle (1pt);

 \filldraw [red] (0.4,.3) circle (2pt);
  \filldraw [blue] (0.8,.3) circle (2pt);
 \filldraw [red] (1.2,.3) circle (2pt);
 
 \filldraw [red] (2.8,.3) circle (2pt);
  \filldraw [blue] (3.2,.3) circle (2pt);
 \filldraw [red] (3.6,.3) circle (2pt);
 
  \filldraw [blue] (5.6,.3) circle (2pt);
 \filldraw [red] (6.0,.3) circle (2pt);
 \filldraw [blue] (6.4,.3) circle (2pt);
 \filldraw [red] (6.8,.3) circle (2pt);

\end{tikzpicture}
\caption{Coloring $\chi_0$, where $a$ is odd, and integers that are $1,3,5,\dots,a-2 \pmod a$ or $a \pmod{2a}$ are red, and integers that are $2,4,6\dots,a-1 \pmod a$ or $0 \pmod{2a}$ are blue.}
\label{coloring_x+y=az}
 \end{figure}

To prove Theorem \ref{upperbx+y=az} we use Lemma \ref{Lemma_nonmono_counting_pairs}. Given a $k$-coloring $\chi$ of $[n]$, we denote by $\nu_\chi(\E,n,k)$ the number of non-monochromatic solutions to $\E$. Note that $\mu_\chi(\E,n,k) +\nu_\chi(\E,n,k) = T_\E(n)$. 
\begin{lemma}\label{Lemma_nonmono_counting_pairs}
    Let $\E$ be an equation of the form $ax + by = cz$. For any $2$-coloring $\chi$ of $[n]$, the number  $\nu_\chi(\E,n,2)$ of non-monochromatic solutions to $\E$ over $[n]$  is precisely $\nu_\chi(\E,n,2) = \frac{1}{2}(\D_{xy}+\D_{xz}+\D_{yz})$, where $$\D_{vv'} = |\{(i,j) : \chi(i) \neq \chi(j), \text{ and there is a solution to } \E \text{ with } v = i, v' = j\}|.$$ 
\end{lemma}
\begin{proof}
    For any non-monochromatic solution $(x,y,z)$ to $\E$, precisely two of the pairs $(x,y), (x,z)$, and $(y,z)$ are non-monochromatic. 
\end{proof}

\begin{reptheorem} {upperbx+y=az}
Let ${\cal E}$ be the equation $x+y=az$. The following holds.

\[ M_{\cal E}(n)\leq \begin{cases}
     \dfrac{n^2}{4a^2} +O(n) & a=3,5, \\
     \\
      \dfrac{8(2a - 1)n^2}{a^4(4 + a)} + O(n)& a=4, a\ge 6,\\
   \end{cases}
\]


but this bound is not tight.
\end{reptheorem}

\begin{proof}[Proof of Theorem \ref{upperbx+y=az}]


  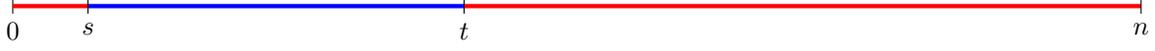
\begin{figure}
    \centering
\begin{tikzpicture}

\draw (0,0.1)--(0,-0.1) node[black,below]{$0$};
\draw [red, ultra thick] (0,0)--(1,0);
\draw [blue, ultra thick] (1,0)--(6,0);
\draw [red, ultra thick] (6,0)--(15,0);

\draw (1,0.1)--(1,-0.1) node[black,below]{$s$};
\draw (6,0.1)--(6,-0.1) node[black,below]{$t$};
\draw (15,0.1)--(15,-0.1) node[black,below]{$n$};

\end{tikzpicture}
\caption{Coloring $\chi$, where the interval $[1,s] \cup [t,n]$ is red and $(s,t)$ blue, where $s = \frac{ 4n(a + 1)}{a^2(4 + a)}, \quad t = \frac{2n}{a}.$}
\label{coloring_block_x+y=az}
 \end{figure}

 Let $a$ be odd and consider the coloring $\chi_0$ that colors integers that are $1,3,5,\dots,a-2 \pmod a$ or $a \pmod{2a}$ red, and integers that are $2,4,6\dots,a-1 \pmod a$ or $0 \pmod{2a}$ blue. Observe that if $(x,y,z)$ is a monochromatic solution to $\E$, then we must have $x,y \equiv 0 \pmod{a}$. Then $z$ is even, so $z \equiv 2,4\dots,2a \pmod{2a}$. Since $a \ge 3$, we have $ 1 \le z = (x+y)/a \le n$ for all $x,y \equiv 0 \pmod{a}$. So for a fixed $x_0 \equiv 0 \pmod a$, there is a solution $(x_0,y,(x_0+y)/a)$ for all $y \in [n]$ with $ y \equiv 0 \pmod a$, and the values $z = (x_0+y)/a$ are equidistributed (up to a constant difference) $\pmod {2a}$.

Therefore if $x \equiv a \pmod{2a}$ is red, there are $\frac{a-1}{2a}\frac{n}{2a} + O(1)$ values $y$ that yield a monochromatic solution. Since there are $\frac{n}{2a} + O(1)$ such values $x$, we have $\frac{(a-1)n^2}{8a^3}+O(n)$ red monochromatic solutions. Similarly, if $x \equiv a \pmod{2a}$ is blue, there are $\frac{a+1}{2a}\frac{n}{2a} + O(1)$ values $y$ that yield a monochromatic solution and we obtain $\frac{(a+1)n^2}{8a^3} + O(n)$ blue monochromatic solutions. Then there are $\frac{n^2}{4a^2} + O(n)$ monochromatic solutions in total. 
 \indent
 Next consider the following coloring $\chi$, where we color the interval $[1,s] \cup [t,n]$ red and $(s,t)$ blue, where $$s = \frac{ 4n(a + 1)}{a^2(4 + a)}, \quad t = \frac{2n}{a}.$$ To count the number of monochromatic solutions, we will use Lemma \ref{Lemma_nonmono_counting_pairs}.  Since $a \ge 3$, we have that $(x,y)$ is a solution to $\E$ if and only if $x + y \equiv 0 \pmod{a}$. For $0\le i \le a-1$, let $r_i$ and $b_i$ denote the number of red and blue integers in $[n]$ congruent to $i \pmod{a}$, respectively, and set $r_a = r_0$ and $b_a = b_0$. Let $R$ and $B$ denote the total number of red and blue integers, respectively. Then we obtain $$\D_{xy} = 2\sum_{i= 0}^{a-1} r_i b_{a-i} \ge \frac{2RB}{a} - O(n) = \frac{2(n-t+s)(t-s)}{a} - O(n)$$ since $r_i \ge R/a -O(1)$ and $b_i \ge B/a - O(1)$ for all $i$. Now to find $\D_{xz} = \D_{yz}$, observe that all such pairs $(x,z)$ that appear in a solution to $\E$ lie in the following region: 
 \begin{center}
     \begin{tikzpicture}
  \draw[fill=gray!50!white] plot[smooth,samples=100,domain=0:8] (\x,1/4*\x) -- 
plot[smooth,samples=100,domain=8:0] (\x,1/4*\x + 2);
\draw (0,0)--(8,0) node[right]{$x$};
 \draw (0,0)--(0,8) node[above]{$z$};
 \draw (8,2pt)--(8,-2pt) node[below] {$n$};
\draw (1.25,2pt)--(1.25,-2pt) node[below] {$s$};
 \draw (2pt,1.25)--(-2pt,1.25) node[left] {$s$};
 \draw (2pt,8)--(-2pt,8) node[left] {$n$};
 \draw (2pt,2)--(-2pt,2) node[left] {$\frac{n}{a}$};
\draw (2pt,4)--(-2pt,4) node[left] {$t = \frac{2n}{a}$};
\draw (4,2pt)--(4,-2pt) node[below] {$t$};
\draw [thick,dotted] (2pt,1.25)--(5,1.25);
\draw [thick,dotted] (1,1)--(1,1);
\draw [thick,dotted] (1.25,0.3)--(1.25,2.3);
\draw [thick,dotted] (4,3) -- (4,1);
\draw [red,thick] (0,0)--(0,1.25); 
\draw [red,thick] (1.25,0)--(0,0); 
\draw [blue,thick] (1.25,0)--(4,0); 
\draw [blue,thick] (0,1.25)--(0,4); 
\draw [red,thick] (0,4)--(0,8); 
\draw [red,thick] (8,0)--(4,0); 
\draw node at (0.62,1.6) {$A_{12}$};
\draw node at (2.5,0.9) {$A_{21}$};
\draw node at (6,2.5) {$A_{32}$};
\end{tikzpicture}
\end{center} 
Let $I_1 = [0,s], I_2 = [s,t], I_3 = [t,n]$, and let $A_{ij}$ denote the area of $I_i \times I_j$ that lies in the shaded paralellogram above (note that the scale is not accurate, but we do have $as \ge t$). Then the number of non-monochromatic pairs $\D_{xz}$ is $A_{12} + A_{21} + A_{23} + A_{32} - O(n)$, the linear term accounting for the boundaries of the rectangles. A straightforward calculation gives \begin{align*}
    A_{12} &= \frac{1}{2} s\left(\frac{n}{a}-s+\frac{n}{a}+\frac{s}{a}-s\right), \\
    A_{21} &= \frac{1}{2} \left(s-\frac{s}{a}+s- \frac{t}{a}\right)(t-s), \\
    A_{23} &= 0, \\
    A_{32} & = \frac{1}{2}(as-t)\left(\frac{n}{a}+\frac{t}{a}-s+\frac{n}{a}\right) + (n-as)\frac{n}{a}. 
\end{align*}
Then we have $$\mu_\chi(\E,n,2) = \frac{n^2}{a} - \frac{1}{2}(\D_{xy} + 2\D_{xz}) \le \frac{8(2a - 1)n^2}{a^4(4 + a)} + O(n).$$

For equation $x+y=az$, we conducted experiments for $n=25,50,75,100$ when $3\leq a \leq7$. Let ${\cal E}: x+y=az$ and let $G_{\cal E}(n,a)= \lfloor\frac{8(2a - 1)n^2}{a^4(4 + a)}\rfloor$. See Table \ref{table:x+y=6z} for a full comparison when $a=6$ and see Appendix \ref{amyIPexperiments} for results when $a=3$. Notice there are a few differences that show that Theorem \ref{upperbx+y=az} is not tight.


\begin{table}[H]
\centering
\begin{tabular}{|p{0.1\textwidth}|p{0.15\textwidth}|p{0.15\textwidth}|p{0.15\textwidth}|p{0.15\textwidth}|}
           \hline
           $n$&    $25$&  $50$& $75$& $100$\\
           \hline
           $M_{\cal E}(n)$&   $3$&  $11$& $27$& $49$\\
           \hline
           $G_{\cal E}(n,6)$&   $4$ &  $16$& $38$& $67$\\
           \hline
          $T_{\cal E}(n)$&  $104$&  $416$&  $937$&  $1667$               \\
           \hline
           \textbf{Optimal Coloring}&  $0^{2}1^{3}0^41^2$ $0^{4}1^30^{2}1^40$&  $10^{3}1^{12}0^{24}$&  $0^{2}1^{4}0^{18}1^{51}$& $1^{2}0^{6}1^{24}0^{68}$\\
           \hline
      
\end{tabular}
\caption{Results for ${\cal E}: x+y=6z$.}
\label{table:x+y=6z}
\end{table}

\end{proof}

For equation the $x+y=2z$, better known as van der Waerden's equation, we conducted experiments when $a=15,20,26,30$ and obtained the best colorings using integer programming. These colorings seem to follow a very particular block pattern that is not easily predictable. Note that these optimal colorings for small $n$ are close to the upper bounds given in \cite{ParriloRobertsonSaracinoMonochromaticAP}. 


\begin{table}[H]
\centering
\begin{tabular}{|p{0.1\textwidth}|p{0.15\textwidth}|p{0.15\textwidth}|p{0.15\textwidth}|p{0.15\textwidth}|}
           \hline
           $n$&    $15$&  $20$& $26$& $30$\\
           \hline
           $M_{\cal E}(n)$&   $25$&  $44$& $74$& $98$\\
           \hline
          $T_{\cal E}(n)$&  $113$&  $200$&  $338$&  $450$               \\
           \hline
           \textbf{Optimal Coloring}&  $0101^20^3$ $1^3010^2$&  $1^20^21^20^4$ $1^40^21^20^2$&  $1^20^21^30^5$ $1^60^31^20^3$&   $0^41^20^31^6$ $0^61^30^2101^2$\\
           \hline
      
\end{tabular}
\caption{Results for ${\cal E}: x+y=2z$.}
\label{table:x+y=2z}
\end{table}

We explore the more general equation $ax+by=z$ and give some general upper bounds.

\begin{reptheorem}{upperbax+by}
     Let $a$ and $b$ be positive integers with $a<b$ and $gcd(a,b)=1$. Let ${\cal E}$ be an equation of the form $ax+by=z$ and let $n\geq R_2 ({\cal E})$. 
    The following holds

    $$M_{\cal E}(n)\leq \sum_{z=a+b}^{\left\lfloor\frac{n}{a(a+b)}\right\rfloor} \left( \frac{z}{ab} - \left\{ \frac{b^{-1}z}{a} \right\} - \left\{\frac{a^{-1}z}{b} \right\} +1 \right),$$

where $\{x\}=x-\lfloor x \rfloor$, $b^{-1}$ is any integer such that $b^{-1}b\equiv 1\pmod a$ and $a^{-1}$ is any integer such that $a^{-1}a\equiv 1 \pmod b$. This bound is not tight.
\end{reptheorem}



    

   \begin{proof}[Proof of Theorem \ref{upperbax+by}]

       Let ${\cal E}: ax+by=z$ where $gcd(a,b)=1$ and $a,b\geq 2$ are distinct integers. Let $\chi$ be a coloring of $[n]$ such that $\chi(i)$ is red if $i\in [1, \lfloor\frac{n}{a(a+b)}\rfloor]\cup [\lfloor \frac{n}{a}\rfloor +1,n]$ and $\chi(i)$ is blue otherwise. See Figure \ref{coloring_ax+by}.
       
       We prove that $$\mu_{\chi}({\cal E},n,2)= \sum_{z=a+b}^{\lfloor\frac{n}{a(a+b)}\rfloor} \left( \frac{z}{ab} - \left\{ \frac{b^{-1}z}{a} \right\} - \left\{\frac{a^{-1}z}{b} \right\} +1 \right),$$ where $\{x\}=x-\lfloor x \rfloor$, $b^{-1}$ is an integer such that $b^{-1}b\equiv 1\pmod a$ and $a^{-1}$ is an integer such that $a^{-1}a\equiv 1 \pmod b$.
       
       Note that any blue choice of $x^*$ or $y^*$ must satisfy 

       $$\left\lfloor\frac{n}{a(a+b)}\right\rfloor +1 \leq x^* , y^* \leq \left\lfloor \frac{n}{a}\right\rfloor.$$
       
       Therefore any blue pair $(x^*, y^*)$ must satisfy that 
       $$\left(a+b\right)\left(\left\lfloor\frac{n}{a(a+b)}\right\rfloor+1\right)\leq ax^* + by^* \leq \left(a+b\right)\left\lfloor\frac{n}{a}\right\rfloor,$$ 
       
       which implies that all corresponding values $z^*$, such that $ax^*+by^*=z^*$, must be red since all integers greater than $\left\lfloor\frac{n}{a}\right\rfloor$ are red. Hence,

       $$\left\lfloor\frac{n}{a}\right\rfloor < \left(a+b\right) \left(\left\lfloor\frac{n}{a(a+b)}\right\rfloor+1 \right) \leq z^*.$$ Therefore, there can only be red solutions in this coloring.

       \begin{observation}
           There cannot be any red solutions where $x^*$ or $y^*$ belong to $[\lfloor \frac{n}{a}\rfloor +1,n]$ because any red value of $x^*$ and $y^*$ is at most $\lfloor \frac{n}{a}\rfloor$. 
       \end{observation}

       Hence, if $(x^* , y^* , z^*)$ is a red triple such that $z^* \in [\lfloor \frac{n}{a}\rfloor +1,n]$, then $x^*,y^*$ must belong to $[1, \lfloor\frac{n}{a(a+b)}\rfloor]$. However, if $x^* , y^* \in [1, \lfloor\frac{n}{a(a+b)}\rfloor]$ are red, then $ax^* + by^* = z^*\leq (a+b)\lfloor \frac{n}{a(a+b)}\rfloor < \lfloor \frac{n}{a}\rfloor+1$. 
       
       Therefore,  

        \begin{observation}
           Red triples $(x^* , y^* , z^*)$ such that $z^* \in [\lfloor \frac{n}{a}\rfloor +1,n]$ 
           cannot exist, and any monochromatic solution $(x^*,y^*,z^*)$ must lie in $[1,\lfloor \frac{n}{a(a+b)}\rfloor]$.
       \end{observation}
       
        It is sufficient to count the number of integer points in the region of $\mathbb{Z}^2$ surrounded by 
        
        $$0<x\leq \left\lfloor\frac{n}{a(a+b)}\right\rfloor, \hspace{3pt} 0<y\leq \left\lfloor\frac{n}{a(a+b)}\right\rfloor,$$ and $$ax+by= \left \lfloor\frac{n}{a(a+b)}\right\rfloor.$$ 
        
        We can do this by counting the number of integer points in each segment $ax+by=z$ contained in this region. As we noted before, any red value of $z^*$ can be found in $[1, \lfloor\frac{n}{a(a+b)}\rfloor]$, therefore we proceed to count the number of integer points in each segment $ax+by=z$ for $0<x\leq \lfloor\frac{n}{a(a+b)}\rfloor$, $0<y\leq \lfloor\frac{n}{a(a+b)}\rfloor$, and $a+b\leq z \leq \lfloor\frac{n}{a(a+b)}\rfloor$. We use a classic result from Ehrhart theory called Popoviciu's theorem \cite{beck2007computing} to count these points,

       $$\sum_{z=a+b}^{\lfloor\frac{n}{a(a+b)}\rfloor} P_{a,b}(z) = \sum_{z=a+b}^{\lfloor\frac{n}{a(a+b)}\rfloor} \# \left\{ (k,l)\in \mathbb{Z}^2 \mid k,l\geq 0, ak+bl=z \right\} = \sum_{z=a+b}^{\lfloor\frac{n}{a(a+b)}\rfloor} \left( \frac{z}{ab} - \left\{ \frac{b^{-1}z}{a} \right\} - \left\{\frac{a^{-1}z}{b} \right\} +1 \right),$$
    
    which gives us the desired result.

To prove that this bound is not tight in general, we conducted experiments using integer programming for various values of $a$, $b$ and $n \geq R_2(ax+by=z)$ (see \cite{WJWThesis}). Let ${\cal E}: ax+by=z$ and let $G_{\cal E}(n,a,b)$ be the upper bound given in Theorem \ref{upperbax+by}.  See Table \ref{table:2x+3y=z} for a full comparison of our results and the values given by the upper bound when $a=2$ and $b=3$. See Appendix \ref{amyIPexperiments} for the remaining experiments.

\begin{table}[H]
\centering
\begin{tabular}  {|p{0.1\textwidth}|p{0.15\textwidth}|p{0.15\textwidth}|p{0.15\textwidth}|p{0.15\textwidth}|}
           \hline
           $n$&  75&  100&  150& 200\\
           \hline
           $M_{\cal E}(n)$&  2&  4&  12& 24\\
           \hline
           $G_{\cal E}(n,2,3)$&   $4$ &  $9$& $23$& $40$\\
           \hline
           $T_{\cal E}(n)$&  444&  800&  1825& 3867\\
           \hline
           \textbf{Optimal Coloring}&  $0^{7}1^{29}01^{2}0$ $10^{34}$&  $0^{9}1^{40}0^{51}$&  $0^{14}1^{59}0^{2}10^{74}$& $0^{19}1^{80}0^{101}$\\
           \hline
      
\end{tabular}
\caption{Results for ${\cal E}: 2x+3y=z$.}
\label{table:2x+3y=z}
\end{table}

   \end{proof}

\begin{figure}[h!]
    \centering
\begin{tikzpicture}

\draw (0,0.1)--(0,-0.1) node[black,below]{$0$};
\draw [red, ultra thick] (0,0)--(1,0);
\draw [blue, ultra thick] (1,0)--(5,0);
\draw [red, ultra thick] (5,0)--(15,0);

\draw (1,0.1)--(1,-0.1) node[black,below]{$\frac{n}{a(a+b)}$};
\draw (5,0.1)--(5,-0.1) node[black,below]{$\frac{n}{a}+1$};
\draw (15,0.1)--(15,-0.1) node[black,below]{$n$};

\end{tikzpicture}
\caption{Coloring $\chi$ used in Theorem \ref{upperbax+by} where $[1, \lfloor\frac{n}{a(a+b)}\rfloor]$ and $[\lfloor \frac{n}{a}\rfloor +1,n]$ are red blocks and the rest is blue.}
\label{coloring_ax+by}
 \end{figure}
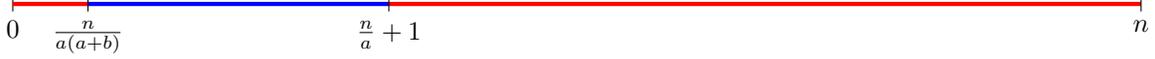






\section{Max-Cut and Best Colorings via Integer Optimization} \label{methodsofIPSDP}

In this section we describe two approaches that allow us to compute optimal or quasi-optimal colorings. We deal mostly with the case of 2 colors, but the theory can be naturally extended to more colors. The key idea is to model this problem as a Max-Cut problem in a graph that one reads 
from the solutions. Recall that for a weighted graph $G= (V,E,w)$, a \emph{maximum cut} of $G$ is 

\begin{align*}
    MAXCUT(G) := \max_{S \subseteq V} \sum_{i \in S, j \in V \setminus S} w_{ij}. 
\end{align*}

A well-known integer programming formulation of the Max-Cut problem is 
\begin{equation} \label{MAXCUT_IP}
    MAXCUT(G) = Z_{IP} := \max \frac 12 \sum_{i < j} w_{ij}(1-x_i x_j), \\
       \quad  x_i \in \{-1,1\} \quad \forall i \in V.  
\end{equation}

The graphs we use are based on the linear equation of study. For a linear equation $\E$, let $G_{\E,n}$ be the weighted graph with vertex set $[n]$ and edge weights $w_{ij}$ for $i \neq j$. Each weight $w_{ij}$ is the number of times $\{i,j\}$ appear together in a solution to the equation $\E$. In general we will denote the adjacency matrix of $G_{\E,n}$ by $C_{\E,n}$.  
For example, the equation $x+y=z$ when $n=5$ is associated to the graph $G_{x+y = z,5}$ which has five vertices.  There is an edge from $i$ to $j$ when they appear together as either $\{x,y\}, \{x,z\}$, or $\{y,z\}$ in a solution $(x,y,z)$ to $ax+by = cz$ (see Figure \ref{fig:weightedgraph}). Note that the solution $(1,1,2)$ contributes $2$ to the weight $w_{1,2}$. 
Thus $G_{x+y = z,5}$ has the following adjacency matrix.

$$C_{x+y = z,5} = \begin{pmatrix} 0 & 4 & 4 & 4 & 2 \\
4 & 0 & 4 & 2 & 2 \\
4 & 4 & 0 & 2 & 2 \\
4 & 2 & 2 & 0 & 2 \\
2 & 2 & 2 & 2 & 0 \\
\end{pmatrix}.$$

\begin{figure}[h]
    \centering
    \includegraphics[scale = 0.5,width=0.4\linewidth]{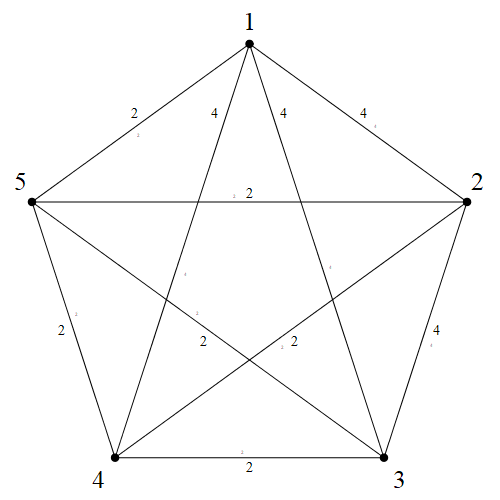}
    \caption{The weighted graph $G_{x+y=z,5}$.}
    \label{fig:weightedgraph}
\end{figure}

For a given $k$-coloring $\chi$ of $[n]$, we denote by $\nu_\chi(\E,n,k)$ the number of non-monochromatic solutions to $\E$. The key idea is that we can bound the number of non-monochromatic solutions in terms of the max-cut of the graph $G_{\E,n}$. We can clearly extend this further.
\begin{lemma}
    Let $\E$ be an equation of the form $ax+ by = cz$. The number of non-monochromatic solutions to $\E$ in a $2$-coloring of $[n]$ is at most $MAXCUT(G_{\E,n})$. 
\end{lemma}
\begin{proof}
    For an arbitrary coloring of $[n]$, if $i$ and $j$ receive different colors, then $x_i$ and $x_j$ have opposite signs in \eqref{MAXCUT_IP}. If a solution to $ax+by = cz$ is non-monochromatic, then exactly two of $\{x,y\}$, $\{x,z\}$, and $\{y,z\}$ are non-monochromatic. Therefore if $i$ and $j$ have different colors, this contributes exactly $w_{ij}$ to the cut, and $$\nu_\chi(\E,n,2) = \frac 12 \sum_{i<j, \chi(i) \neq \chi(j)} w_{ij} \le \frac 12 MAXCUT(G_{\E, n}).$$
\end{proof}

The integer program in \eqref{MAXCUT_IP} can be solved using standard tools from optimization such as branch-and-bound and cutting planes as implemented in GUROBI \cite{gurobi}. 

In our research, instead of using the quadratic model \eqref{MAXCUT_IP}, we use an alternative linear model to compute optimal colorings given the graph $G_{\E,n}$. Concretely, the model we use is as follows:
\begin{align*}
\text{maximize}\ &\sum_{(i,j)\in E}w_{ij}e_{ij} \\
\text{subject\ to}\ e_{ij}&\leq x_i+x_j\ &\forall(i,j)\in E \\
e_{ij}&\leq 2-(x_i+x_j)\ &\forall(i,j)\in E \\
x_i,x_j&\in \{0,1\}\ &\forall(i,j)\in E \\
e_{ij}&\in \{0,1\}\ &\forall(i,j)\in E 
\end{align*}

The computation we did using integer programs provided optimal 2-colorings for all equations in Section \ref{theupperbounds} and several values of $n$. These optimal values helped us to predict upper bounds and show that the theorems in Section \ref{theupperbounds} are not tight in general. 



As an alternative, we turn to the following \emph{semidefinite relaxation} of \eqref{MAXCUT_IP}. 

\begin{equation}\label{MAXCUT_SDP}
 Z_{SDP} = \max \frac 12 \sum_{i < j} w_{ij}(1-X_{ij}), \\
       \quad  X_{ii} =1 \quad \forall i \in V, X = (X_{ij}), X \succeq 0.
\end{equation}

This relaxation was studied in \cite{GoemansWilliamson} and can be used to give a good approximation algorithm for the Max-Cut problem. We have that $Z_{SDP} \ge Z_{IP}$, but in general this inequality is strict. The semidefinite program above can be solved for a fixed $n$, but does not immediately yield any asymptotics for the number of monochromatic solutions. In the next section we solve a different semidefinite program model, to reach exact bounds.

\section{Optimal Colorings with a Semidefinite Program Trace Bound Method} \label{jacktracesdp}

 The authors in \cite{ParriloRobertsonSaracinoMonochromaticAP} used the following lemma to give lower bounds on the number of monochromatic arithmetic progressions in a 2-coloring of $[n]$.
\begin{lemma} \label{Lemma_QuadraticFormPSDBound}
    If $A$ is an $n \times n$ matrix and $x \in [-1,1]^n$, then if $D = diag(d_1,d_2,\dots,d_n)$ is such that $A+D$ is positive semidefinite, then $x^T A x \ge -\sum_{i=1}^n d_i$. 
\end{lemma}
This matrix $A+D$ can be found using semidefinite programming. To illustrate the method, we first give an example that is another proof that $M_{\text{Schur}}(n) \ge \frac{n^2}{11}$. Our initial analysis is essentially the same as in \cite{ThanatipanondaSchurTriplesProblem2009}, but we arrive at the conclusion by counting monochromatic solutions indirectly with the use of Lemma \ref{Lemma_nonmono_counting_pairs}. 

\begin{example}\label{ExampleSchur}
Let $\chi$ be an arbitrary 2-coloring of $[n]$, and let $\E$ be the Schur equation $x + y = z$. Observe that $\mu_\chi(\E,n,2) = T_\E(n) - \nu_\chi(\E,n,2)$. For $x+y = z$, we have $\mathcal S = n^2/2 - O(n)$. From Lemma \ref{Lemma_nonmono_counting_pairs}, write $\nu_\chi(\E,n,k) = \frac 12 (\D_{xy} + \D_{xz} + \D_{yz})$, where $$\D_{xy} = |\{(i,j) : i = x, j = y \text{ for some solution to } (x,y,z) \text{ and } i,j \text{ have different colors}\} |.$$ 

We have that $(x,z) = (i,j)$ is part of a solution to $x+y = z$ if and only if $i<j$. For each red-blue pair, exactly one integer is larger, so we have $\D_{xz} = \D_{yz} = |R||B|$, where $R$ is the set of red integers and $B$ is the set of blue integers. We have that $(x,y) = (i,j)$ is part of a solution to $x+y = z$ if and only if $i+j \le n$. 

Now subdivide $[n]$ into $k$ intervals, and let $r_i,b_i$ denote the number of red and blue integers, respectively, in interval $i$. Following \cite{ThanatipanondaSchurTriplesProblem2009}, let $S_{ij} = r_ib_j + r_jb_i$ be the number of dichromatic pairs in the square $I_i \times I_j$. 
Then, we have that $\D_{xy} \le \sum_{i+j \le k+1} S_{ij}$ and $\D_{xz} = \D_{yz} = (\sum_{i=1}^k r_i)(\sum_{i=1}^k b_i)$. 

We need to make one additional refinement. For $i+j = k+1$, we are overcounting. For $i$ and $j$ where we expect $r_i = r_j = n/k$ and $b_i = b_j = 0$, this does not matter, but if $r_i = b_j = n/k$, this is a problem. We instead bound the number by half the area of the square $I_i \times I_j $, which is $\frac{n^2}{2k^2}$ (we can do this since half the pairs $(x,y)$ in that square do not form solutions because $x+y > n)$.

Since we know a priori that we should have $r_i = n/k$ for $i \in \{1,2,3,4,11\}$ and $r_i = 0$ otherwise, we bound $S_{ij}$ by $n^2/242$ for $(i,j) \in X:= (2,10),(3,9),(4,8),(8,4),(9,3),(10,2)$. 

Putting it all together gives $$\mathcal \nu_\chi(\E,n,k) \le \frac12 \left( 2 (\sum_{i=1}^k r_i)(\sum_{i=1}^k b_i) + \sum_{i+j \le k+1, (i,j) \not \in X } S_{ij} + 6n^2/242\right).$$

Now we make the change of variable $r_i =\frac{(1+x_i)}{2}\frac{n}{11}$, $b_i =\frac{(1-x_i)}{2}\frac{n}{11}$. 

This gives $\nu_\chi(\E,n,k) \le \frac{17n^2}{44} +\frac{n^2}{121}Q$, where $Q$ is a quadratic form. We can then bound $Q$ using Lemma \ref{Lemma_QuadraticFormPSDBound}.

Write $Q = x^T A x$. Then we have $$\mu_\chi(\E,n,k) \ge \frac{n^2}{2} - \frac{17n^2}{44} + (x^T(-A)x)\frac{n^2}{121} - O(n).$$

We can check that $-A + diag(d)$ is positive semidefinite where $d = (\frac{1}{2},\frac 12, \frac 14, 0,0,\frac{1}4,\frac 12, \frac 12,\frac 14,0,0)$. Then $$\mu_\chi(\E,n,k) \ge \frac{n^2}{2} - \frac{17n^2}{44} -\frac{11}4 \frac{n^2}{121} = \frac{n^2}{11} -O(n)$$ as desired.
\end{example}

 The general method is as follows: partition the interval $[n]$ into blocks $P_1,\dots, P_k$. Note that in Example \ref{ExampleSchur} the blocks were subintervals, but this need not be the case in general. Each block $P_i$ contains some number of red integers $r_i$ and blue integers $b_i$. For a sensible partition, the number of non-monochromatic pairs can be bounded by a quadratic form in $r_i$ and $b_i$. Then under the change of variable $$r_i = \frac{(1+x_i)|P_i|}{2}, b_i = \frac{(1-x_i)|P_i|}{2},$$ the number of nonmonochromatic solutions is a quadratic form in $x_i$, where $x_i \in [-1,1]$. The following lemma allows us to bound this quadratic form (see \cite{ParriloRobertsonSaracinoMonochromaticAP}).
 
We can apply this method to any equation, though its usefulness varies. There are three possibilities: \emph{tight} bounds, \emph{nontrivial} bounds, and \emph{trivial} bounds. The best case is that we obtain a tight lower bound that matches the upper bound from a given coloring. Example \ref{ExampleSchur} shows that we get a tight lower bound for the Schur equation. We will show that we obtain tight bounds for other equations as well. In \cite{ParriloRobertsonSaracinoMonochromaticAP}, the bound for the equation $x+y = 2z$ obtained   by this method is close, but not equal to, the best known upper bound (which is conjectured to be optimal). They remark that in general there is a gap between the SDP relaxation and the original problem. In this particular instance, they partitioned $[n]$ into 128 subintervals to bound their quadratic form. This overcounts the number of non-monochromatic solutions, and more intervals will give better bounds. However, the bounds given by this method may not converge to the optimum. In this case we obtain a \emph{nontrivial} (and still useful), but not tight bound. Unfortunately, for some equations $\E$, we get the useless bound $M_\E(n) \ge f(n)$ for some $f(n) \le 0$. This appears to occur for equation $4x+4y = z$, for example. 

Given an equation, it is not immediately clear which type of bounds this method gives, nor is it clear how many intervals or how fine a partition is necessary to obtain tight bounds. It would be interesting to investigate this phenomenon further in future work and determine which equations give tight or nontrivial bounds.

\begin{remark}
Although this technique does not work for the equation $4x+4y = z$, it does work for the equation $4x-4y = z$. For the second equation, coloring multiples of 4 red and all other integers blue is optimal. This same coloring is conjectured to be optimal for the first equation as well, but the trace method appears to give trivial bounds. However, if we consider both equations together, that is, if we minimize the total number of monochromatic triples $(x,y,z)$ that are solutions to \emph{either} equation, we do obtain a tight bound. 
\end{remark}

We can now apply this technique to equations of the form $ax - ay = z$ to prove Theorem \ref{Theorem_ax_ay_z} and obtain \emph{tight} lower bounds for small values of $a$. 

\begin{reptheorem} {Theorem_ax_ay_z}
Let ${\cal E}:ax-ay = z$ and $n\geq R_2 ({\cal E})$. For $3\le a \le 7$, the following holds 
$$M_{\cal E}(n) \ge \frac{(2a-1)n^2}{2a^4},$$ 
where an optimal coloring is given by coloring all multiples of $a$ red and every other integer blue. 
\end{reptheorem}

\begin{proof}[Proof of Theorem \ref{Theorem_ax_ay_z}]


For this lower bound, we carry out similar analysis as we did for the Schur equation $x+y = z$. Here we let $\E$ be the equation $a(x-y) = z$. We again subdivide $[n]$ into $k$ equally sized intervals $I_1,\dots, I_n$. Let $\chi$ be an arbitrary 2-coloring of $[n]$. Let $r^*_i$ denote the number of integers in $I_i$ that are colored red and are not multiples of $a$. Let $r^a_i$ denote the number of integers in $I_i$ that are colored red and are multiples of $a$. Define $b^*_i$ and $b^a_i$ similarly for blue integers. Clearly we have \begin{align*}r^*_i+b^*_i &= \left(1-\frac 1a \right)\frac{n}k,\\
r^a_i+b^a_i &= \frac{n}{ak},
\end{align*}
for all $i$. 

The next step is to find expressions for $\D_{xy}, \D_{xz}$, and $\D_{yz}$ in terms of $r^*_i,r^a_i,b^*_i$, and $b^a_i$. 

For $\D_{xy}$, we have already seen that $(x,y)$ appears in a solution $(x,y,z)$ if and only if $ 1 \le x-y \le n/a$. Therefore all such pairs $(x,y)$ solutions to $\E$ lie in the shaded region below. 

\begin{center}
\begin{tikzpicture}
 \draw[fill=gray!50!white] plot[smooth,samples=100,domain=0:3] (\x,{\x}) -- 
     plot[smooth,samples=100,domain=3:1] (\x,{\x-1});
\draw[samples=100, domain=1:3] plot (\x,{\x-1});
\draw[samples=100,domain=0:3] plot (\x,{\x}) ;
\draw[samples=100,domain=2:3] plot (3,{\x}) ;
\draw (0,0)--(3.5,0) node[right]{$x$};
 \draw (0,0)--(0,3.5) node[above]{$y$};

 \draw (3,2pt)--(3,-2pt) node[below] {$n$};
 \draw (2pt,3)--(-2pt,3) node[left] {$n$};
 \draw (1,2pt)--(1,-2pt) node[below] {$\frac{n}{a}$};
\end{tikzpicture}

\end{center} 
For each square $I_i \times I_j$, we bound the number of nonmonochromatic pairs $(x,y)$ in the square by 

$$S^{xy}_{ij} := \begin{cases}
    0 & \text{ if } I_i \times I_j \text{ does not intersect the region}, \\
    (r^*_i+r^a_i)(b^*_j+b^a_j)+(b^*_i+b^a_i)(r^*_j+r^a_j) & \text{ if $I_i \times I_j$ is completely in the interior of the region}, \\
    r^*_i b^*_j +r^a_i b^a_j + b^*_i r^*_j + b^a_i r^a_j + \frac 1 a\left(1-\frac 1a\right) \frac{n^2}{k^2} & \text{otherwise.}
\end{cases}$$
In the last case observe that exactly half the area of the square is contained in the region. Then the total number of pairs $(x,y)$ in the square where one of $x$ and $y$ is a multiple of $a$ and the other is not is $\frac 2 a\left(1-\frac 1a\right) \frac{n^2}{k^2}$, and half of these are inside the region. (Note also that if we knew, a priori, that the multiples of $a$ coloring is optimal, then the remaining terms in the bound are all 0.)

We proceed with similar analyses for $\D_{xz}$ and $\D_{yz}$. Observe that $(x,z)$ appears in a solution $(x,y,z)$ to $\E$ if and only if $z$ is a multiple of $a$ and $z \le ax$. Then, all such pairs $(x,z)$ lie in the shaded region below. 

\begin{center}
\begin{tikzpicture}
 \draw[fill=gray!50!white] plot[smooth,samples=100,domain=0:3] (\x,0) -- 
     plot[smooth,samples=100,domain=3:1] (\x,3);
\draw[samples=100, domain=1:3] plot (\x,3);
\draw[samples=100,domain=0:1] plot (\x,{3*\x}) ;
\draw[samples=100,domain=0:3] plot (3,{\x}) ;
\draw (0,0)--(3.5,0) node[right]{$x$};
 \draw (0,0)--(0,3.5) node[above]{$z$};

 \draw (3,2pt)--(3,-2pt) node[below] {$n$};
 \draw (2pt,3)--(-2pt,3) node[left] {$n$};
 \draw (1,2pt)--(1,-2pt) node[below] {$\frac{n}{a}$};
\end{tikzpicture}
\end{center} 
We again bound the number of nonmonochromatic pairs $(x,z)$ in the square $I_i \times I_j$ by
$$S^{xz}_{ij} := \begin{cases}
    0 & \text{ if } I_i \times I_j \text{ does not intersect the region}, \\
    (r^*_i+r^a_i)b^a_j+(b^*_i+b^a_i)r^a_j & \text{ if $I_i \times I_j$ is completely in the interior of the region}, \\
     r^a_i b^a_j +b^a_i r^a_j + \frac{A^{xz}_{ij}}{a}\left(1-\frac 1a\right) \frac{n^2}{k^2} & \text{otherwise},
\end{cases}$$
where $A^{xz}_{ij}$ is the fraction of the area of the square $I_i \times I_j$ contained in the region. 

For $\D_{yz}$, observe that $(y,z)$ is a solution to $\E$ if and only if $z$ is a multiple of $a$ and $ay + z \le na$. So the solutions are in the shaded region below. 

\begin{center}
\begin{tikzpicture}
 \draw[fill=gray!50!white] plot[smooth,samples=100,domain=0:3] (\x,0) -- 
     plot[smooth,samples=100,domain=2:0] (\x,3);
\draw[samples=100, domain=0:2] plot (\x,3);
\draw[samples=100,domain=2:3] plot (\x,{9-3*\x}) ;

\draw (0,0)--(3.5,0) node[right]{$y$};
 \draw (0,0)--(0,3.5) node[above]{$z$};

 \draw (3,2pt)--(3,-2pt) node[below] {$n$};
 \draw (2pt,3)--(-2pt,3) node[left] {$n$};
 \draw (1,2pt)--(1,-2pt) node[below] {$\frac{n}{a}$};
\end{tikzpicture}
\end{center} 

The bounds are similar to the $xz$-case: we bound the number of non-monochromatic $(y,z)$ in the square $I_i \times I_j$ by
$$S^{yz}_{ij} := \begin{cases}
    0 & \text{ if } I_i \times I_j \text{ does not intersect the region}, \\
    (r^*_i+r^a_i)b^a_j+(b^*_i+b^a_i)r^a_j & \text{ if $I_i \times I_j$ is completely in the interior of the region}, \\
     r^a_i b^a_j +b^a_i r^a_j + \frac{A^{xz}_{ij}}{a}\left(1-\frac 1a\right) \frac{n^2}{k^2} & \text{otherwise,}
\end{cases}$$
where $A^{yz}_{ij}$ is the fraction of the area of the square $I_i \times I_j$ contained in the region. 

Putting it all together, by Lemma \ref{Lemma_nonmono_counting_pairs} we have \begin{align*} 
\mu_\chi(\E,n,k) &= T_\E(n) - \nu_\chi(\E,n,k) \\
 & \ge \frac{2a-1}{2a^2}n^2 - \frac12\left(\sum_{1 \le j \le k}S^{xy}_{ij}+S^{xz}_{ij}+S^{yz}_{ij}\right).
\end{align*}
Next, we make the following variable substitutions:
\begin{align*}
    r^*_i &= \frac{1+x_i}{2}\left(1-\frac{1}{a}\right)\frac{n}{k}, &r^*_a = \left(\frac{1+y_i}{2}\right)\frac{n}{ak}, \\ 
    b^*_i &= \frac{1-x_i}{2}\left(1-\frac{1}{a}\right)\frac{n}{k}, &b^*_a = \left(\frac{1-y_i}{2}\right)\frac{n}{ak}.
\end{align*}
Observe that $-1 \le x_i,y_i \le 1$ for all $i$. After this substitution we obtain the following: 

$$ \mu_\chi(\E,n,k) \ge \left(\frac{2a-1}{2a^2}-\alpha_{a,k}\right)n^2 + \textbf{x}^TA\textbf{x},$$
where $\alpha_{a,k}$ is a constant depending only on $a$ and $k$ and $\textbf{x}$ is the vector $(x_1,\dots,x_k,y_1,\dots,y_k)$. 

To obtain a bound on $\textbf{x}^TA\textbf{x}$, we used Lemma \ref{Lemma_QuadraticFormPSDBound} and found vectors $d$ such that $A+diag(d)$ is positive semidefinite using the SDP solver SeDuMi \cite{SeDuMi}. For $4 \le a \le 7$, we used $a^2$ subintervals. For $a=3$ this was insufficient and we used $k=54$ subintervals. We obtained $\alpha_{3,54} = \frac{277}{1296},\alpha_{4,16} = \frac{93}{512}, \alpha_{5,25} = \frac{73}{500}, \alpha_{6,36} = \frac{211}{1728},$ and $\alpha_{7,49} = \frac{36}{343}$. The precise values for $d$ are given in the Appendix \ref{Appendix_SPD_lower_bounds}. In each case, the given lower bound matches the upper bound $\mu_\chi(\E,n,k) \le \frac{2a-1}{2a^4}$. 
\end{proof}


We now perform similar analysis for the equation $x + y = 3z$. 

\begin{reptheorem} {Theorem_x_y_3z}
Let ${\cal E}:x+y=3z$ and $n\geq R_2 ({\cal E})$. The following bound holds

$$M_{\cal E} (n)\geq \frac{167n^2}{28800}.$$
\end{reptheorem}

\begin{proof}[Proof of Theorem \ref{Theorem_x_y_3z}]
    Let $\E$ be the equation $x + y = 3z$, and let $\chi$ be a 2-coloring of $[n]$. To count the total number of solutions to $\E$, observe that for all $x,y \in [n]$, we have $\frac{x+y}{3} \in [n]$ if and only if $x + y \equiv 0 \pmod 3$. Now for a given $x$, there are $\frac{n}{3} - O(1)$ possible values for $y \in [n]$ such that $x + y \equiv 0 \pmod 3$. Then there are $n\left(\frac{n}{3} - O(1)\right) = \frac{n^2}{3} - O(n)$ solutions to $\E$. 

Next we count non-monochromatic solutions to $\E$ as in Lemma \ref{Lemma_nonmono_counting_pairs}. As above, we see that $(x,y)$ appears in a solution $(x,y,z)$ if and only if $x + y \equiv 0 \pmod 3$. We have that $(x,z)$ appears in a solution $(x,y,z)$ if and only if $ \frac{x}{3} \le z \le \frac{x}{3} + \frac{n}{3}$, that is, $(x,z)$ lies in the shaded region below. 

\begin{center}
\begin{tikzpicture}
 \draw[fill=gray!50!white] plot[smooth,samples=100,domain=0:3] (\x,1/3*\x) -- 
     plot[smooth,samples=100,domain=3:0] (\x,1/3*\x + 1);
\draw (0,0)--(3.5,0) node[right]{$x$};
 \draw (0,0)--(0,3.5) node[above]{$z$};

 \draw (3,2pt)--(3,-2pt) node[below] {$n$};
 \draw (2pt,3)--(-2pt,3) node[left] {$n$};
 \draw (2pt,1)--(-2pt,1) node[left] {$\frac{n}{3}$};
\draw (2pt,2)--(-2pt,2) node[left] {$\frac{2n}{3}$};
\end{tikzpicture}
\end{center} 

For this equation we partition the interval $[n]$ into $k$ subintervals of length $\frac{n}{k}$. For $1\le i \le k$ and  $0 \le j \le 5$, let $r_{i}^j$ and $b_i^j$ denote the number of integers in subinterval $i$ congruent to $j \pmod 3$ that are colored red and blue, respectively. The number of non-monochromatic pairs $(x,y)$ is 

$$\mathcal{D}_{xy} = 2\sum_{i = 1}^k [r_i^0b_j^0+ r_i^1b_j^2 + r_i^2b_j^1] = 2(R^0B^0+R^1B^2+R^2B^1),$$
where $R^i = \sum_{j=1}^k r_j^i$ and $B^i = \sum_{j=1}^k b_j^i$. Note that $B^i = \frac{n}{3} - R^i$ and $0 \le R^i \le \frac{n}{3}$ for all $i$. It is not difficult to show via calculus that 
\begin{equation}\label{Dxy_bound}
    \D_{xy} \le \frac{5n^2}{18} 
\end{equation} and a global maximum occurs at $(R^0,R^1,R^2) = (\frac{n}{6},0,\frac{n}{3}).$

We then approximate the number of non-monochromatic pairs $(x,z)$ by $ \D_{xz} \le  \sum_{1\le i<j \le k} S_{ij}^{xz}$, where 
$$S_{ij}^{xz} = 
    \begin{cases}
    0 & \text{ if } I_i \times I_j \text{ does not intersect the region}, \\
   (\sum_{l=0}^2 r_i^l)(\sum_{l=0}^2 b_j^l) & \text{ otherwise}. \\
\end{cases}$$
From the symmetry of $x$ and $y$ in $\E$, we have $\D_{xz} = \D_{yz}$
Then, making the substitutions
$$ r_i^j = \frac{1+x_i^j}{2} \frac{n}{3k}, \hspace{10pt} b_i^j = \frac{1-x_i^j}{2} \frac{n}{3k},$$
setting $k = 30$ and using the bound \eqref{Dxy_bound}, we obtain
\begin{align*}
    \mu_\chi(\E,n,k) &\ge \frac{n^2}{3} - \frac{1}{2} (\D_{xy} + 2\D_{xz}) \\  & \ge \left(\frac{1}{3}-\frac{5}{36}-\frac{11}{60} \right)n^2 + \textbf{x}^T A\textbf{x},
\end{align*} 
where $\textbf{x}$ is the vector 


$$(x_1^0,\dots,x_k^0,x_1^1,\dots,x_k^1,x_1^2,\dots,x_k^2).$$

We then bound the quadratic form using Lemma \ref{Lemma_QuadraticFormPSDBound}. We obtained a $d$ such that $A+ diag(d)$ is positive semidefinite such that $\sum_{i= 1}^{90} d_i \le \frac{1377}{259200}$. We give the precise value of this $d$ in Appendix \ref{Appendix_SPD_lower_bounds}. 

\end{proof}
    
The lower bound in Theorem \ref{Theorem_x_y_3z} is far from the best known upper bound. It is computationally feasible to use more intervals to obtain a tighter bound, and different partition choices may yield better results as well. 

\section{More than two colors in Schur's equation} \label{sectionmorethan2colors}
Here we discuss briefly minimizing the number of monochromatic solutions to Schur's equation, $x+y = z$, for more than two colors. First, we give the colorings used to prove Theorem \ref{TheoremMuUpperBounds}. 

\begin{reptheorem}{TheoremMuUpperBounds}
    The following improved bounds on $M_{\emph{Schur}}(n,k)$ hold:

    $$M_{\emph{Schur}}(n,3) \le \frac{n^2}{67} + O(n), \hspace{20 pt } M_{\emph{Schur}}(n,4) \le \frac{n^2}{496} + O(n).$$
\end{reptheorem}

\begin{proof}[Proof of Theorem \ref{TheoremMuUpperBounds}]

A 3-coloring with $\frac{n^2}{67}$ monochromatic solutions to $x+y = z$ is $$0^{\frac{10n}{67}}1^{\frac{14n}{67}}0^{\frac{2n}{67}}2^{\frac{28n}{67}}0^{\frac{n}{67}}1^{\frac{11n}{67}}0^{\frac{n}{67}}$$

\begin{center}
\begin{tikzpicture}

\draw (0,0.1)--(0,-0.1) node[black,below]{0};
\draw [red, ultra thick] (0,0)--(16*10/67,0);

\draw (16*10/67,0.1)--(16*10/67,-0.1) node[black,below]{$\frac{10n}{67}$};
\draw [blue, ultra thick] (16*10/67,0)--(16*10/67+16*14/67,0);
\draw (16*10/67+16*14/67,0.1) --  (16*10/67+16*14/67,-0.1);
\draw [dashed](16*10/67+16*14/67,-0.15) -- (16*10/67+16*14/67-0.25,-0.5)node[black, below] { $\frac{24n}{67}$};
\draw [red, ultra thick] (16*10/67+16*14/67,0)--(16*26/67,0) ;
\draw [dashed](16*26/67,-0.15) -- (16*26/67+0.25,-0.5)node[black, below] { $\frac{26n}{67}$};
\draw (16*26/67,0.1) --  (16*26/67,-0.1);
\draw [green, ultra thick] (16*54/67,0)--(16*26/67,0);
\draw [dashed](16*55/67,-0.15) -- (16*55/67+0.25,-0.5)node[black, below] { $\frac{55n}{67}$};
\draw (16*55/67,0.1) --  (16*55/67,-0.1);

\draw (16*54/67,0.1) --  (16*54/67,-0.1);
\draw [dashed](16*54/67,-0.15) -- (16*54/67-0.25,-0.5)node[black, below] { $\frac{54n}{67}$};

\draw (16*66/67,0.1) --  (16*66/67,-0.1);
\draw [dashed](16*66/67,-0.15) -- (16*66/67-0.25,-0.5)node[black, below] { $\frac{66n}{67}$};

\draw [dashed](16*67/67,-0.15) -- (16*67/67+0.25,-0.5)node[black, below] { $n$};
\draw (16*67/67,0.1) --  (16*67/67,-0.1);

\draw [red, ultra thick] (16*54/67,0)--(16*55/67,0);
\draw [blue, ultra thick] (16*66/67,0)--(16*55/67,0);
\draw [red, ultra thick] (16*66/67,0)--(16*67/67,0);
\end{tikzpicture}
\end{center} 

A 4-coloring with $\frac{n^2}{496}$ monochromatic solutions to $x+y = z$ is $$0^{\frac{28}{496}}1^{\frac{38}{496}}0^{\frac{5}{496}}2^{\frac{75}{496}}0^{\frac{2}{496}}1^{\frac{30}{496}}0^{\frac{2}{496}}3^{\frac{182}{496}}0^{\frac{1}{496}}1^{\frac{29}{496}}0^{\frac{1}{496}}2^{\frac{72}{496}}0^{\frac{1}{496}}1^{\frac{29}{496}}0^{\frac{1}{496}}0^{\frac{1}{496}}$$

\begin{center}
\begin{tikzpicture}
\draw (0,0.1)--(0,-0.1) node[black,below]{0};
\draw [red, ultra thick] (0,0)--(16*28/496,0);
\draw [blue, ultra thick] (16*28/496,0)--(16*66/496,0);
\draw [red, ultra thick] (16*66/496,0)--(16*71/496,0);
\draw [green, ultra thick](16*71/496,0)--(16*146/496,0);
\draw [red, ultra thick] (16*146/496,0)--(16*148/496,0);
\draw [blue, ultra thick] (16*148/496,0)--(16*178/496,0);
\draw [red, ultra thick] (16*178/496,0)--(16*180/496,0);
\draw [yellow, ultra thick] (16*180/496,0)--(16*362/496,0);
\draw [red, ultra thick] (16*362/496,0)--(16*363/496,0);
\draw [blue, ultra thick] (16*363/496,0)--(16*392/496,0);
\draw [red, ultra thick] (16*392/496,0)--(16*393/496,0);
\draw [green, ultra thick](16*393/496,0)--(16*465/496,0);
\draw [red, ultra thick](16*465/496,0)--(16*466/496,0);
\draw [blue, ultra thick] (16*466/496,0)--(16*495/496,0);
\draw [red, ultra thick] (16*495/496,0)--(16*496/496,0);

\draw [black] (16*28/496,0.1)--(16*28/496,-0.1) node[black, below] { $\frac{28n}{496}$};
\draw [black] (16*66/496,0.1)--(16*66/496,-0.1);
\draw [black] (16*71/496,0.1)--(16*71/496,-0.1);
\draw [black](16*146/496,0.1)--(16*146/496,-0.1);
\draw [black] (16*148/496,0.1)--(16*148/496,-0.1);
\draw [black] (16*178/496,0.1)--(16*178/496,-0.1);
\draw [black] (16*180/496,0.1)--(16*180/496,-0.1);
\draw [black] (16*362/496,0.1)--(16*362/496,-0.1);
\draw [black] (16*363/496,0.1)--(16*363/496,-0.1);
\draw [black] (16*392/496,0.1)--(16*392/496,-0.1);
\draw [black] (16*393/496,0.1)--(16*393/496,-0.1);
\draw [black](16*465/496,0.1)--(16*465/496,-0.1);
\draw [black](16*466/496,0.1)--(16*466/496,-0.1);
\draw [black] (16*495/496,0.1)--(16*495/496,-0.1);
\draw [black] (16*496/496,0.1)--(16*496/496,-0.1);

\draw [dashed](16*495/496-.02,-0.15) -- (16*495/496-0.25,-0.5)node[black, below] { $\frac{495n}{496}$};
\draw [dashed](16*496/496+.02,-0.15) -- (16*496/496+0.25,-0.5)node[black, below] {$n$ };

\draw [dashed](16*465/496-.02,-0.15) -- (16*465/496-0.65,-0.5)node[black, below] { $\frac{465n}{496}$};
\draw [dashed](16*466/496+.02,-0.15) -- (16*466/496+0.02,-0.5)node[black, below] {$\frac{466n}{496}$};

\draw [dashed](16*392/496-.02,-0.15) -- (16*392/496-0.25,-0.5)node[black, below] { $\frac{392n}{496}$};
\draw [dashed](16*393/496+.02,-0.15) -- (16*393/496+0.5,-0.5)node[black, below] {$\frac{393n}{496}$ };

\draw [dashed](16*362/496-.02,-0.15) -- (16*362/496-0.65,-0.5)node[black, below] { $\frac{362n}{496}$};
\draw [dashed](16*363/496+.02,-0.15) -- (16*363/496+0.02,-0.5)node[black, below] {$\frac{363n}{496}$};

\draw [dashed](16*178/496-.02,-0.15) -- (16*178/496-0.25,-0.5)node[black, below] { $\frac{178n}{496}$};
\draw [dashed](16*180/496+.02,-0.15) -- (16*180/496+0.5,-0.5)node[black, below] {$\frac{180n}{496}$ };

\draw [dashed](16*146/496-.02,-0.15) -- (16*146/496-0.65,-0.5)node[black, below] { $\frac{146n}{496}$};
\draw [dashed](16*148/496+.02,-0.15) -- (16*148/496+0.02,-0.5)node[black, below] {$\frac{148n}{496}$};

\draw [dashed](16*66/496-.02,-0.15) -- (16*66/496-0.25,-0.5)node[black, below] { $\frac{66n}{496}$};
\draw [dashed](16*71/496+.02,-0.15) -- (16*71/496+0.25,-0.5)node[black, below] {$\frac{71n}{496}$ };
\end{tikzpicture}
\end{center} 

\end{proof} 

 Recall that an optimal coloring $\chi_2$ for $k = 2$ colors is the optimal coloring where the first $\frac{4n}{11}$ integers are the first color, the next $\frac{6n}{11}$ the second, and the last $\frac{n}{11}$ the third. We represent this optimal coloring as $\chi_2 := 0^{4n/11}1^{6n/11}0^{1/11}$. For a given $k$-coloring $ \chi = a_1^{e_1} a_2^{e_2} \cdots a_j^{e_j}$, (where $a_i \in [k]$ and $e_i \in \Z^+$)  define its \emph{block pattern} to be the word $a_1a_2\cdots a_j$. We say that a block pattern for a $k$-coloring is \emph{greedy-palindromic} if it is of the form $P_k$, where $P_1 = 0$, $P_{k} = P_{k-1}(k-1)P_{k-1}$ for $k \ge 2$. Observe that $P_k = a_1\cdots a_{2^k-1}$ has length $2^k-1$ and that $a_j = \nu_2(j)$, where $\nu_2(j)$ denotes the 2-adic valuation of $j$, the highest power of $2$ that divides $j$. 

Greedy palindromic colorings are not new to the study of Schur numbers. they can be used to give the lower bound $S(k) \ge 3S(k-1)-1$ (see \cite{BBGeneralizedSchur}). Though this is not tight for $k > 3$. Our method for finding these new colorings was to optimize, for each $k$, a certain polynomial $p_k$ whose variables are the lengths of the intervals in a greedy palindromic coloring. The optimal solutions minimize the total number of monochromatic solutions where, for each monochromatic solution $(x,y,z)$ in color $c$, at least one of $x,y$, or $z$ comes from the \emph{first} block of color $c$. The polynomial $p_k$ is defined as follows. 

$$p_k = \sum_{i = 1}^{2^k-1} C_i(x_i - s_{a_i})^2,$$ where 
$$C_i = \begin{cases}
    \frac 12 & \text{if } i = 2^j \text{ for some } j \ge 0 \\
    1 & \text{otherwise.}
\end{cases},\qquad  s_{a_i} = \sum_{j=1}^{2^{a_i}-1} x_j. $$

The optimization problem is then 
\begin{align}\label{SchurPolyOptimization}
    &\text{minimize } p_k \\
    &\nonumber\text{subject to }  \sum_{i=1}^{2^k-1}x_i = n.
\end{align}

For $k = 2$, the optimal solution to \eqref{SchurPolyOptimization} gives the optimal coloring. For $k = 3,4$ we obtain the colorings used in Theorem \ref{TheoremMuUpperBounds}, and $p_k$ here counts the total number of monochromatic solutions.

When $k \ge 5$, we can of course still optimize $p_k$. However, $p_k$ no longer counts the total number of monochromatic solutions, because there are solutions that do not use integers from the first block in that color. 


Given a $(2^{k}-1)$-tuple  $(e_1,\dots, e_{2^k-1})$ and integer $n$, let $M = \sum_{i=1}^{2^{k}-1} e_i$, and let 
$C_{P_k}(e_1,\dots,e_{2^k-1})$ denote the coloring $a_1^{e_1n/M} \cdots a_{2^k-1}^{e_{2^k-1}n/M}$, where $P_k = a_1\cdots a_{2^k-1}$ as above. In this notation, we have $\chi_2 = C_{P_2}(4,6,1)$. 
For $3 \le k \le 8$, the optimal solutions to \label{SchurPolyOptimization} are the following:

\begin{align*}
    \chi_3 &= C_{P_3}(10,14,2,28,1,11,1), \\
    \chi_4 &= C_{P_4}(28,38,5,75,2,30,2,182,1,29,1,72,1,29,1),\\
    \chi_5 &= C_{P_5}(82, 110, 14, 216, 5, 87, 5, 523, 2, 84, 2, 208, 2, 84, 2, 1428, 1, 83, 1, 207, 1, 83, 1, 520, 1, 83, 1, 207, 1, 83, 1),\\
\chi_6 &= C_{P_6}(244, 326, 41, 639, 14, 258, 14, 1546, 5, 249, 5, 616, 5, 249,   5, 4220, 2, 246, 2, 613, 2, 246, 2, 1538, 2, 246, 2,\\&\hspace{35pt} 613, 2, 246, 2, 12202, 1, 245, 1, 612, 1, 245, 1, 1537, 1, 245, 1, 612, 1, 245, 1, 4217, 1, 245, 1, 612, 1, 245, 1,\\&\hspace{35pt} 1537, 1, 245, 1, 612, 1, 245, 1) \\
\chi_7 &= C_{P_7}(730, 974, 122, 1908, 41, 771, 41, 4615, 14, 744, 14, 1840, 14, 744, 14, 12596, 5, 735, 5, 1831, 5, 735, 5, \\&\hspace{35pt} 4592, 5, 735, 5, 1831, 5, 735, 5, 36420, 2, 732, 2, 1828, 2, 732, 2, 4589, 2, 732, 2, 1828, 2, 732, 2, 12588, 2, \\&\hspace{35pt} 732, 2,  1828, 2, 732, 2, 4589, 2, 732, 2, 1828, 2, 732, 2, 107804, 1, 731, 1, 1827, 1, 731, 1, 4588, 1, 731, 1,\\&\hspace{35pt} 1827, 1, 731, 1, 12587, 1, 731, 1, 1827, 1, 731, 1, 4588, 1, 731, 1, 1827, 1, 731, 1, 36417, 1, 731, 1, 1827, 1,\\&\hspace{35pt} 731, 1, 4588, 1, 731, 1, 1827, 1, 731, 1,  12587, 1, 731, 1, 1827, 1, 731, 1, 4588, 1, 731, 1, 1827, 1, 731, 1) \\
\chi_8 &= C_{P_8}(
2188, 2918, 365, 5715, 122, 2310, 122, 13822, 41, 2229, 41, 5512, 41, 2229, 41, 37724, 14, 2202, 14, 5485, \\&\hspace{35pt}
14, 2202, 14, 13754, 14, 2202, 14, 5485, 14, 2202, 14, 109074, 5, 2193, 5, 5476, 5, 2193, 5, 13745, 5, 2193,\\&\hspace{35pt} 
5, 5476, 5, 2193, 5, 37701, 5, 2193, 5, 5476, 5, 2193, 5, 13745, 5, 2193, 5, 5476, 5, 2193, 5, 322861, 2, 2190, \\&\hspace{35pt} 
2, 5473, 2, 2190, 2, 13742, 2, 2190, 2, 5473, 2, 2190, 2, 37698, 2, 2190, 2, 5473, 2, 2190, 2, 13742, 2, 2190,\\&\hspace{35pt} 
2, 5473, 2, 2190, 2, 109066, 2, 2190, 2, 5473, 2, 2190, 2, 13742, 2, 2190, 2, 5473, 2, 2190, 2, 37698, 2, 2190, \\&\hspace{35pt} 
2, 5473, 2, 2190, 2, 13742, 2, 2190, 2, 5473, 2, 2190, 2,964038, 1, 2189, 1, 5472, 1, 2189, 1,\\&\hspace{35pt} 
13741, 1, 2189, 1, 5472, 1, 2189, 1, 37697, 1, 2189, 1, 5472, 1, 2189, 1, 13741, 1, 2189, 1, 5472, 1, 2189, 1,\\&\hspace{35pt} 
109065, 1, 2189, 1, 5472, 1, 2189, 1, 13741, 1, 2189, 1, 5472, 1, 2189, 1, 37697, 1, 2189, 1, 5472, 1, 2189, 1,\\&\hspace{35pt} 
13741, 1, 2189, 1, 5472, 1, 2189, 1, 322858, 1, 2189, 1, 5472, 1, 2189, 1, 13741, 1, 2189, 1, 5472, 1, 2189, 1,\\&\hspace{35pt} 
37697, 1, 2189, 1, 5472, 1, 2189, 1, 13741, 1, 2189, 1, 5472, 1, 2189, 1, 109065, 1, 2189, 1, 5472, 1, 2189, 1,\\&\hspace{35pt} 
13741, 1, 2189, 1, 5472, 1, 2189, 1, 37697, 1, 2189, 1, 5472, 1, 2189, 1, 13741, 1, 2189, 1, 5472, 1, 2189, 1)
\end{align*}






These colorings appear to follow a predictable pattern, which we conjecture holds for all $k$. 
For a given color $c \in [k]$, let $S_c := \sum_{i=1}^{2^{c-1}-1}e_i$, that is the sum of the $e_i$ before color $c$ appears. Let $c_{max}(j) = \lfloor{\log_2(j)}\rfloor +1 = \max_{i<j} a_i$ denote the largest color that has appeared before the $j$-th interval. 

$$e_j(k)  := \begin{cases}
    S_i + 3^{k-i}+1 &\text{ if } j = 2^{i-1}, \\
    S_c + \frac{3^{k-c_{max}(j)}+1}{2} & \text{ if $a_j = c$ and $j \neq 2^{i-1}$.} 
\end{cases}.$$

For $k = 5$, the coloring $\chi_5$ has $\frac{4128+22^2}{4128^2}n^2 = \frac{1153}{4260096}n^2$ solutions, which is more than the number from the 5-coloring obtained by Thanatipanonda \cite{ThanatipanondaSchurTriplesProblem2009}. It is unclear whether the colorings $\chi_k$ are optimal for $k \ge 3$. Although for $k = 3$, there is some experimental evidence for optimality. Colorings found by local search methods similar to those done in \cite{ButlerCostelloGrahamConstellations} are similar to $\chi_3$. We leave further analysis of $\chi_k$ for $k \ge 3$ and determining whether these colorings are optimal as an open question.  






\vskip .5cm

\section*{Acknowledgements} 
The research of the first three authors was partially supported by the NSF grant DMS-2348578. The authors thank Adriana Hansberg and Amanda Montejano for suggestions and ideas at the beginning of this project. The second author is grateful for the financial support received from the UC Davis Chancellor's Postdoctoral Fellowship Program. She is also grateful to Adriana Hansberg and Amanda Montejano for their mentoring during her doctoral and master's studies. 

\bibliographystyle{abbrv}
\bibliography{bibliography}

\begin{appendices}
    \section{Data for SDP lower bounds}\label{Appendix_SPD_lower_bounds}
    In this appendix, we explicitly show the values for $d$ using Lemma \ref{Lemma_QuadraticFormPSDBound} for the equation $ax-ay = z$.\\

    $a=3: d = \frac{n^2}{209952}(62, 66, 70, 74, 78, 82, 86, 90, 94, 98, 102, 106, 110, 114, 118, 122, 126, 130, 132, 132, 132,\\ 132, 132, 132, 132, 132, 132, 132, 132, 132, 132, 132, 132, 132, 132, 132, 130, 126, 122, 118, 114, 110, 106, 102,\\ 98, 94, 90, 86, 82, 78, 74, 70, 66, 62, 61, 59, 57, 53, 51, 49, 45, 43, 41, 37, 35, 33, 29, 27, 25, 21, 19, 17, 14, 14, 14,\\ 12, 12, 12, 10, 10, 10, 8, 8, 8, 6, 6, 6, 4, 4, 4, 3, 5, 7, 7, 9, 11, 11, 13, 15, 15, 17, 19, 19, 21, 23, 23, 25, 27)$

    $a=4: d = \frac{n^2}{32768}(3,9,15,21,24,24,24,24,24,24,24,24,21,15,9,3,41,39,37,35,30,30,30,30,\\26,26,26,26,23,25,27,29)$

    $a =5: d= \frac{n^2}{125000}(4, 12, 20, 28, 36, 40, 40, 40, 40, 40, 40, 40, 40, 40, 40, 40, 40, 40, 40, 40, 36, 28, 20, 12, 4,\\  121, 119, 117, 115, 113, 106, 106, 106, 106, 106, 100, 100, 100, 100, 100, 94, 94, 94, 94, 94, 89, 91, 93, 95,97)$

    $a=6: d = \frac{n^2}{373248}(5, 15, 25, 35, 45, 55, 60, 60, 60, 60, 60, 60, 60, 60, 60, 60, 60, 60, 60, 60, 60, 60, 60, 60, 60,\\ 60, 60, 60, 60, 60, 55, 45, 35, 25, 15, 5, 253, 251, 249, 247, 245, 243, 234, 234, 234, 234, 234, 234, 226, 226, 226,\\ 226, 226, 226, 218, 218, 218, 218, 218, 218, 210, 210, 210, 210, 210, 210, 203, 205, 207, 209, 211, 213)$

    $a=7: d= \frac{n^2}{941192}(6, 18, 30, 42, 54, 66, 78, 84, 84, 84, 84, 84, 84, 84, 84, 84, 84, 84, 84, 84, 84, 84, 84, 84, 84, 84,\\ 84, 84, 84, 84, 84, 84, 84, 84, 84, 84, 84, 84, 84, 84, 84, 84, 78, 66, 54, 42, 30, 18, 6, 449, 447, 445, 443, 441, 439, 437,\\ 426, 426, 426, 426, 426, 426, 426, 416, 416, 416, 416, 416, 416, 416, 406, 406, 406, 406, 406, 406, 406, 396, 396, 396, \\396, 396, 396, 396, 386, 386, 386, 386, 386, 386, 386, 377, 379, 381, 383, 385, 387, 389)$
    
    The following is the value of $d$ for $x+y = 3z:$ 

\begin{align*}
    d &= \frac{n^2}{259200}(1.52625029424193,7.13472273267136,9.82265291183994,7.83419180674646,13.2109296915416,\\&20.8751792783345,15.3089585102576,9.25870080147831,8.19058277713545,13.1809852450228,\\&13.1809852450181,16.0328894562698,6.91545792002705,0.599907563422803,9.50693006172301,\\&10.5063647943975,19.1826045564940,15.8643109427448,9.77858541708185,13.6613915855486,\\&19.9455937540945,24.4274671277702,24.4274671277680,24.4274671277802,25.4460632299112,\\&25.4460632299175,25.4460632299234,22.6076114659320,22.6076114659365,22.6076114659281,\\&1.52625029424280,7.13472273266820,9.82265291183396,7.83419180674915,13.2109296915387,\\&20.8751792783391,15.3089585102469,9.25870080147229,8.19058277713694,13.1809852450137,\\&13.1809852450188,16.0328894562674,6.91545792001992,0.599907563423494,9.50693006171557,\\&10.5063647944014,19.1826045564880,15.8643109427641,9.77858541708290,13.6613915855242,\\&19.9455937540848,24.4274671277696,24.4274671277692,24.4274671277664,25.4460632299061,\\&25.4460632299254,25.4460632299143,22.6076114659347,22.6076114659356,22.6076114659403,\\&1.52625029424287,7.13472273266038,9.82265291183561,7.83419180674793,13.2109296915385,\\&20.8751792783367,15.3089585102529,9.25870080147053,8.19058277713283,13.1809852450232,\\&13.1809852450176,16.0328894562679,6.91545792003126,0.599907563422704,9.50693006172368,\\&10.5063647944013,19.1826045564943,15.8643109427621,9.77858541707903,13.6613915855242,\\&19.9455937540863,24.4274671277702,24.4274671277736,24.4274671277731,25.4460632299182,\\&25.4460632299198,25.4460632299157,22.6076114659358,22.6076114659345,22.6076114659353)
\end{align*}

    

\section{Integer Programming Results}
\label{amyIPexperiments}
 
 In this appendix, we provide additional results for equations $ax+ay=z$, $ax-ay=z$, $x+y=az$, and $ax+by=z$ that are not listed in the main content of the paper. To show a $2$-coloring of $[n]$, we use $0$ and $1$ to represent each color in the order they appear in the coloring. The upper index represents the number of times the values within the parenthesis are being colored with those colors and their order. For example, $0^3(10)^2 1^2 0$ represents a $2$-coloring of $[10]$ where the integers are colored in the following order: $0001010110$.\\
 
 In the following tables, we show results for a particular equation ${\cal E}$. We abbreviate by $G_{\cal E}(n,a)$ the value of the upper bound of $M_{\cal E}(n)$ obtained in this work. Similarly, we will use $G_{\cal E}(n,a,b)$ to abbreviate the value of the upper bound of $M_{\cal E}(n)$, given here, when ${\cal E}$ is a three-term linear equation that depends on positive integers $a$ and $b$.\\

\subsection*{Experimental Results for 2-colorings for $ax+ay=z$} 
\begin{table}[H]
\centering
\begin{tabular}  {|p{0.1\textwidth}|p{0.15\textwidth}|p{0.15\textwidth}|p{0.15\textwidth}|p{0.15\textwidth}|}
           \hline
           $n$&  $34$&  $50$&  $75$& $100$\\
           \hline
           $M_{\cal E}(n)$&  $1$&  $2$&  $4$& $7$\\
           \hline
           $G_{\cal E}(n,2)$&   $31$ &  $71$& $166$& $300$\\
           \hline
           $T_{\cal E}(n)$&  $72$&  $156$&  $342$& $625$\\
           \hline
           \textbf{Optimal Coloring}&  $1^{3}0^{14}(10){8}1$&  $1^{7}0^{24}(10)^{9}1$&  $1^{7}010^{28}(10)^{19}$& $1^{11}0^{38}(10)^{25}1$\\
           \hline
      
\end{tabular}
\caption{Results for ${\cal E}: 2x+2y=z$.}
\label{table:2x+2y=z}
\end{table}

\begin{table}[H]
\centering
\begin{tabular}  {|p{0.1\textwidth}|p{0.15\textwidth}|p{0.15\textwidth}|p{0.15\textwidth}|p{0.15\textwidth}|}
           \hline
           $n$&  $111$&  $150$&  $200$& $250$\\
           \hline
           $M_{\cal E}(n)$&  $1$&  $1$&  $2$& $4$\\
           \hline
           $G_{\cal E}(n,3)$&   $69$ &  $130$& $235$& $371$\\
           \hline
           $T_{\cal E}(n)$&  $342$&  $625$&  $1089$& $1722$\\
           \hline
           \textbf{Optimal Coloring}&  $1^{6}0^{32}(100)^{24}1$&  $0^{8}1^{41}010^{99}$&  $0^{11}1^{55}0^{2}10^{131}$& $0^{13}1^{69}0^{168}$\\
           \hline
      
\end{tabular}
\caption{Results for ${\cal E}: 3x+3y=z$.}
\label{table:3x+3y=z}
\end{table}

\subsection*{Experimental Results of 2-colorings for $ax-ay=z$} 

\begin{table}[H]
\centering
\begin{tabular}{|p{0.1\textwidth}|p{0.15\textwidth}|p{0.15\textwidth}|p{0.15\textwidth}|p{0.15\textwidth}|}
           \hline
           $n$&    $25$&  $50$& $75$& $100$\\
           \hline
           $M_{\cal E}(n)$&   $5$&  $30$& $62$& $129$\\
           \hline
           $G_{\cal E}(n,4)$&   $7$ &  $32$& $74$& $133$\\
           \hline
           $T_{\cal E}(n)$&  $129$&  $522$&  $1179$&  $2175$\\
           \hline
           \multicolumn{5}{|c|}{ \textbf{Optimal Coloring:} integers $v\leq n$, get $0$ if $v\equiv 0 \pmod a$, and $1$ otherwise.}\\
           \hline
      
\end{tabular} 
\caption{Results for ${\cal E}: 4x-4y=z$.}
\label{table:4x-4y=z}
\end{table}

\subsection*{Experimental Results of 2-colorings for $x+y=az$}

\begin{table}[H]
\centering
\begin{tabular}{|p{0.1\textwidth}|p{0.15\textwidth}|p{0.15\textwidth}|p{0.15\textwidth}|p{0.15\textwidth}|}
           \hline
           $n$&    $25$&  $50$& $75$& $100$\\
           \hline
           $M_{\cal E}(n)$&   $12$&  $54$& $132$& $231$\\
           \hline
           $G_{\cal E}(n,3)$&   $17$ &  $69$& $156$& $277$\\
           \hline
          $T_{\cal E}(n)$&  $208$&  $834$&  $1875$& $3333$\\
           \hline
           \textbf{Optimal Coloring}& $(010)(011)(010)^3$ $(011)^30$& $(101)^2(100)^3$ $(101)^5(100)^610$& $(011)^3(010)^4$ $(011)^9(010)^9$&  $(010)(011)(010)^3$ $(011)^5(010)^{11}$ $(011)^{12}0$\\ 
           \hline
      
\end{tabular}
\caption{Results for ${\cal E}: x+y=3z$.}
\label{table:x+y=3z}
\end{table}

\subsection*{Experimental Results of 2-colorings for $ax+by=z$}






Tables \ref{table:2x+5y=z}, \ref{table:2x+7y=z}, and \ref{table:3x+5y=z} show results when $gcd(a,b)=1$.

\begin{table}[H]
\centering
\begin{tabular}{|p{0.1\textwidth}|p{0.15\textwidth}|p{0.15\textwidth}|p{0.15\textwidth}|p{0.15\textwidth}|}
           \hline
           \multicolumn{5}{|c|}{$2x+5y=z, \hspace{2pt}\ R_2(2x+5y=z)=103$} \\
           \hline
           $n$&  $200$&  $250$&  $300$& $400$\\
           \hline
           $M_{\cal E}(n)$&  $5$&  $10$&  $15$& $30$\\
           \hline
           $T_{\cal E}(n)$&  $1940$&  $3050$&  $4410$& $7880$\\
           \hline
           \textbf{Optimal Coloring}&  $0^{13}1^{84}(01)^2$ $0^{99}$&  $0^{16}101^{104}(01)^2$ $0^{3}10^{3}10^{114}$&  $1^{20}0^{128}1^{152}$& $1^{27}0^{170}1^{203}$\\
           \hline
      
\end{tabular}
\caption{Results for ${\cal E}: 2x+5y=z$.}
\label{table:2x+5y=z}
\end{table}

\begin{table}[H]
\centering
\begin{tabular}{|p{0.1\textwidth}|p{0.15\textwidth}|p{0.15\textwidth}|p{0.15\textwidth}|p{0.15\textwidth}|}
           \hline
           \multicolumn{5}{|c|}{$2x+7y=z, \hspace{2pt}\ R_2(2x+7y=z)=169$} \\
           \hline
           $n$&  $200$&  $250$&  $300$& $400$\\
           \hline
           $M_{\cal E}(n)$&  $1$&  $3$&  $5$& $10$\\
           \hline
           $T_{\cal E}(n)$&  $1372$&  $2161$&  $3129$& $5600$\\
           \hline
           \textbf{Optimal Coloring}&  $1^{10}0^{86}101^3$ $0101^{96}$&  $0^{13}1^{112}(01)^30^{119}$&  $1^{16}0^{131}1^201^{150}$& $1^{21}0^{176}(10)^31^{197}$\\
           \hline
      
\end{tabular}
\caption{Results for ${\cal E}: 2x+7y=z$.}
\label{table:2x+7y=z}
\end{table}

\begin{table}[H]
\centering
\begin{tabular}{|p{0.1\textwidth}|p{0.15\textwidth}|p{0.15\textwidth}|p{0.15\textwidth}|p{0.15\textwidth}|}
           \hline
           \multicolumn{5}{|c|}{$3x+5y=z, \hspace{2pt}\ R_2(3x+5y=z)=197$} \\
           \hline
           $n$&  $200$&  $250$&  $300$& $400$\\
           \hline
            $M_{\cal E}(n)$&  $1$&  $1$&  $2$& $5$\\
           \hline
           $T_{\cal E}(n)$&  $1287$&  $2025$&  $2930$& $5240$\\
           \hline
           \textbf{Optimal Coloring}&  $0^{10}101^{75}(011)^3$ $0^{5}10^{2}10^510^{89}$&  $0^{10}101^{2}01^{72}01^{3}$ $(01)^2 10^{2}1^{4}(001)^2$ $01(100)^40^{3}$ $10^{8}10^{115}$&  $0^{12}1^{2}01^{85}010^{2}$ $1^{2}0101^{2}0^{2}1^{2}$ $010^{2}10^{5}10^{174}$& $1^{16}0^{119}10^{2}1$ $01^{2}01^{257}$\\
           \hline
      
\end{tabular}
\caption{Results for ${\cal E}: 3x+5y=z$.}
\label{table:3x+5y=z}
\end{table}

Tables \ref{table:2x+4y=z} - \ref{table:3x+6y=z} are results for $gcd(a,b)>1$. We currently do not have a general upper bound for equations of this sort. This is a problem left for future studies.

\begin{table}[H]
\centering
\begin{tabular}{|p{0.1\textwidth}|p{0.15\textwidth}|p{0.15\textwidth}|p{0.15\textwidth}|p{0.15\textwidth}|}
           \hline
           \multicolumn{5}{|c|}{$2x+4y=z, \hspace{2pt}\ R_2(2x+4y=z)=76$} \\
           \hline
           $n$&  $100$&  $150$&  $200$& $400$\\
           \hline
           $M_{\cal E}(n)$&  $2$&  $5$&  $11$& $52$\\
           \hline
           $T_{\cal E}(n)$&  $600$&  $1369$&  $2450$& $9000$               \\
           \hline
           \textbf{Optimal Coloring}&  $1^{8}0^{41}(10)^{50}1$&  $0^{11}1^{60}010^{77}$&  $0^{15}101^{81}0^{102}$& $1^{31}010^{166}(10)^{100}1$\\
           \hline
      
\end{tabular}
\caption{Results for ${\cal E}: 2x+4y=z$.}
\label{table:2x+4y=z}
\end{table}

\begin{table}[H]
\centering
\begin{tabular}{|p{0.1\textwidth}|p{0.15\textwidth}|p{0.15\textwidth}|p{0.15\textwidth}|p{0.15\textwidth}|}
           \hline
           \multicolumn{5}{|c|}{$2x+6y=z, \hspace{2pt}\ R_2(2x+6y=z)=134$} \\
           \hline
           $n$&  $200$&  $250$&  $300$& $400$\\
           \hline
           $M_{\cal E}(n)$&  $3$&  $5$&  $9$& $18$\\
           \hline
          $T_{\cal E}(n)$&  $1617$&  $2542$&  $3675$& $6567$               \\
           \hline
           \textbf{Optimal Coloring}&  $0^{13}1^{85}(01)^3$ $0^31010^{90}$&  $1^{15}0^{108}(10)^{63}1$&  $1^{17}0^{130}(10)^{76}1$& $1^{25}0^{178}10^3(10)^{96}1$\\
           \hline
      
\end{tabular}
\caption{Results for ${\cal E}: 2x+6y=z$.}
\label{table:2x+6y=z}
\end{table}

\begin{table}[H]
\centering
\begin{tabular}{|p{0.1\textwidth}|p{0.15\textwidth}|p{0.15\textwidth}|p{0.15\textwidth}|p{0.15\textwidth}|}
           \hline
           \multicolumn{5}{|c|}{$2x+8y=z, \hspace{2pt}\ R_2(2x+8y=z)=208$} \\
           \hline
           $n$&    $250$&  $300$& $350$& $400$\\
           \hline
           $M_{\cal E}(n)$&   $2$&  $3$& $4$& $6$\\
           \hline
          $T_{\cal E}(n)$&  $1891$&  $2738$&  $3741$&  $4900$               \\
           \hline
           \textbf{Optimal Coloring}&  $1^{13}0^{118}(1000)^2$ $(10)^{55}1$&  $0^{15}1^{131}(01)^{2}$ $0^5(10)^20^{141}$&  $1^{17}0^{162}(10)^{85}1$& $1^{19}0^{180}(10)^{100}1$\\
           \hline
      
\end{tabular}
\caption{Results for ${\cal E}: 2x+8y=z$.}
\label{table:2x+8y=z}
\end{table}

\begin{table}[H]
\centering
\begin{tabular}{|p{0.1\textwidth}|p{0.15\textwidth}|p{0.15\textwidth}|p{0.15\textwidth}|p{0.15\textwidth}|}
           \hline
           \multicolumn{5}{|c|}{$3x+6y=z, \hspace{2pt}\ R_2(3x+6y=z)=249$} \\
           \hline
           $n$&    $250$&  $300$& $350$& $400$\\
           \hline
           $M_{\cal E}(n)$&   $1$&  $1$& $2$& $2$\\
           \hline
          $T_{\cal E}(n)$&  $1681$&  $2450$&  $3306$&  $4536$               \\
           \hline
           \textbf{Optimal Coloring}&  $0^{9}1^{72}0^210^{2}10^{163}$&  $1^{10}0^{88}(100)^{67}1$&  $0^{13}1^{101}0^210^{233}$& $1^{14}0^{117}(100)^{89}10$\\
           \hline
      
\end{tabular}
\caption{Results for ${\cal E}: 3x+6y=z$.}
\label{table:3x+6y=z}
\end{table}
\end{appendices}

\end{document}